\date{22 October, 2012. Revised: 29 November, 2013.}

\documentclass[11pt,reqno]{amsart}
\usepackage{latexsym,amsmath,amsfonts,amscd,amssymb}
\usepackage[all]{xy}

\theoremstyle{plain}
\newtheorem{theorem}{Theorem}[section]
\newtheorem{corollary}[theorem]{Corollary}
\newtheorem{lemma}[theorem]{Lemma}
\newtheorem{proposition}[theorem]{Proposition}
\theoremstyle{definition}
\newtheorem{definition}[theorem]{Definition}

\numberwithin{equation}{section}

\newcommand{\s}{\sigma}

\newcommand{\ba}{\mathbf{a}}

\newcommand{\bn}{\mathbf{n}}

\newcommand{\bT}{\bar{\bar{T}}}

\newcommand{\cA}{\mathcal{A}}

\newcommand{\cD}{\mathcal{D}}
\newcommand{\cE}{\mathcal{E}}
\newcommand{\cF}{\mathcal{F}}

\newcommand{\cM}{\mathcal{M}}

\newcommand{\cN}{\mathcal{N}}

\newcommand{\cO}{\mathcal{O}}

\newcommand{\cR}{\mathcal{R}}
\newcommand{\cS}{\mathcal{S}}
\newcommand{\cT}{\mathcal{T}}
\newcommand{\cU}{\mathcal{U}}

\renewcommand{\AA}{\mathbb{A}}
\newcommand{\QQ}{\mathbb{Q}}
\newcommand{\LL}{\mathbb{L}}
\newcommand{\RR}{\mathbb{R}}
\newcommand{\CC}{\mathbb{C}}

\newcommand{\KK}{\mathbb{K}}
\newcommand{\PP}{\mathbb{P}}
\newcommand{\ZZ}{\mathbb{Z}}
\newcommand{\un}{\boldsymbol{1}}
\newcommand{\RC}{\cR_C}
\newcommand{\wRC}{\widetilde{\cR}_C}

\newcommand{\x}{\times}
\newcommand{\ox}{\otimes}

\newcommand{\frM}{{\frak M}}
\newcommand{\frh}{{\frak h}}
\newcommand{\frs}{{\frak s}}
\newcommand{\frS}{{\frak S}}

\newcommand{\frf}{{\frak f}}
\newcommand{\hs}{{\frh\frs}}

\newcommand{\fhs}{{\frf\frh\frs}}

\newcommand{\Sym}{\mathrm{Sym}}

\newcommand{\Jac}{\mathrm{Jac}\,}
\newcommand{\Bl}{\mathrm{Bl}}
\DeclareMathOperator{\rk}{rk}
\DeclareMathOperator{\codim}{codim} \DeclareMathOperator{\im}{im}

\DeclareMathOperator{\Hom}{Hom} \DeclareMathOperator{\Ext}{Ext}
\DeclareMathOperator{\End}{End} \DeclareMathOperator{\Gr}{Gr}
\DeclareMathOperator{\GL}{GL}\DeclareMathOperator{\Sp}{Sp}
\DeclareMathOperator{\Hg}{Hg}

\newcommand{\Mot}{{\frM}ot}

\newcommand{\scp}{{\s_c^+}}
\newcommand{\smp}{{\s_m^+}}
\newcommand{\scm}{{\s_c^-}}

\newcommand{\VarC}{\mathcal{V}ar_{\CC}}
\newcommand{\SmVarC}{\mathcal{S}m\mathcal{V}ar_{\CC}}

\let\oldmarginpar\marginpar
\renewcommand\marginpar[1]{\oldmarginpar{\tiny\bf\begin{flushleft} #1
\end{flushleft}}}

\title[Motives and Hodge Conjecture for moduli of pairs]{Motives and the Hodge
Conjecture for moduli spaces of pairs}


\author[V. Mu\~noz]{Vicente Mu\~noz}
  \address{Facultad de Matem\'{a}ticas \\ Universidad Complutense
  de Madrid \\ Plaza Ciencias 3
  \\ 28040 Madrid \\ Spain}
  \email{vicente.munoz@mat.ucm.es}

\author[A. G. Oliveira]{Andr\'e G. Oliveira}
  \address{Departamento de Matem\'atica \\ 
  Escola de Ci\^encias e Tecnologia \\ 
  Universidade de Tr\'{a}s-os-Montes e Alto Douro \\
Quinta dos Prados \\ 5001-801 Vila Real \\ Portugal}
  \email{agoliv@utad.pt}

\author[J. S\'anchez]{Jonathan S\'anchez}
  \address{Facultad de Matem\'{a}ticas \\ Universidad Complutense
  de Madrid \\ Plaza Ciencias 3
  \\ 28040 Madrid \\ Spain}
  \email{jnsanchez@mat.ucm.es}

\thanks{First and third authors partially supported by Spanish MICINN grant MTM2010-17389. Second author partially supported by FCT (Portugal) with national funds
through the projects PTDC/MAT/099275/2008, PTDC/MAT/098770/2008 and PTDC/MAT/120411/2010 and through Centro de Matem\'atica da
Universidade de Tr\'as-os-Montes e Alto Douro (PEst-OE/MAT/UI4080/2011). Thanks also to the support by ITGP (Interactions of Low-Dimentional Topology and Geometry with Mathematical Physics), an ESF RNP}

\begin{document}
\maketitle

\begin{abstract}
 Let $C$ be a smooth projective curve of genus $g\geq 2$ over
 $\CC$. Fix $n\geq 1$, $d\in \ZZ$.
 A pair $(E,\phi)$ over $C$ consists of an algebraic vector bundle
 $E$ of rank $n$ and degree $d$ over $C$ and a section $\phi \in H^0(E)$.
 There is a concept of stability for pairs which depends on a
 real parameter $\tau$. Let $\frM_\tau(n,d)$ be the moduli space
 of $\tau$-polystable pairs of rank $n$ and degree $d$ over $C$. 
We prove that for a generic curve $C$, the moduli
space $\frM_\tau(n,d)$ 
satisfies the Hodge Conjecture for $n \leq 4$.
For obtaining this, we prove first that $\frM_\tau(n,d)$ 
 is motivated by $C$.
\end{abstract}

\section{Introduction} \label{sec:introduction}

Let $C$ be a smooth projective curve of genus $g\geq 2$ over the
field of complex numbers. Fix $n \geq 1$ and $d\in\ZZ$, and a line
bundle $L_0$ of degree $d$.
We denote by $M(n,d)$ the moduli space of polystable vector bundles of rank
$n$ and degree $d$ over $C$, and by $M(n,L_0)$
the moduli space of polystable bundles $E$ with
determinant $\det(E)\cong L_0$.

A pair $(E,\phi)$ over $C$ consists of a vector bundle $E$ of rank $n$ and
degree $d$ together with a section $\phi\in H^0(E)$. There
is a concept of stability for a pair which depends on the choice of
a parameter $\tau \in \RR$. This gives a collection of moduli spaces
of $\tau$-polystable pairs $\frM_\tau(n,d)$ and moduli spaces of
pairs with fixed determinant, $\frM_\tau(n,L_0)$, which are projective
varieties. These moduli spaces have been studied by many authors
\cite{B,BD,GP,Mu1,Mu2,MOV1}.
The range of the parameter $\tau$ is an open interval $I$ split by a
finite number of critical values $\tau_c$. For non-critical $\tau$,
$\frM_\tau(n,d)$ is a smooth projective variety.
The main goal of this paper is to prove the following result.

\begin{theorem} \label{thm:main-HC}
  Let $\tau \in I$ be non-critical and $n \leq 4$. Suppose that $C$ is
  a very general curve. Then the moduli spaces $\frM_\tau(n,d)$ 
  satisfy the Hodge Conjecture.
\end{theorem}

The genericity condition in Theorem \ref{thm:main-HC} means 
that $C$ belongs to the complement of a countable union of closed sets of the moduli of curves.
More explicitly, the genericity condition that we need can be stated
as requiring that $C^k$ satisfies the Hodge Conjecture for any $k\geq 1$.

\medskip

In order to prove this result, we shall refine the techniques of \cite{Mu-HS}, where
the first author studied the mixed Hodge structures associated to
these moduli spaces. Here we compute the class defined by the moduli spaces
in the ring of motives $K(\Mot)$. There is a map of abelian groups
 $$
 \Theta: K(\Mot) \to K(\fhs),
 $$
where $\fhs$ is the category of filtered Hodge structures. Moreover, a smooth projective
variety satisfies the Hodge Conjecture if and only if its class in $K(\Mot)$ is contained in the kernel of $\Theta$.
Let $\RC \subset K(\Mot)$ be the subring that contains all iterated
symmetric products and self-products of the curve $C$.
We say that \emph{$X$ is motivated by $C$} if $[X]\in \RC$.
For a generic curve, $\RC\subset \ker \Theta$, and so any
projective smooth variety motivated by $C$ satisfies the Hodge Conjecture.
Here we prove that

\begin{theorem} \label{thm:motivated}
 Suppose $n\leq 4$ and let $\tau$ be generic. For \emph{any} curve $C$,
 the moduli spaces $\frM_\tau(n,d)$ 
 are motivated by $C$. 
\end{theorem}

To find $[\frM_\tau(n,d)] \in K(\Mot)$, we study the behaviour of
the moduli spaces when we change the value of the parameter $\tau$.
When $\tau$ moves without crossing a critical value, the moduli
space remains unchanged. When $\tau$ crosses a critical value,
$\frM_\tau(n,d)$ undergoes a birational
transformation which is called a \emph{flip}. This consists of
removing some subvariety and inserting a different one. A stratification of the flip locus was obtained in \cite{Mu-HS}, and this allows us
to give an explicit geometrical description of flip loci when $n\leq 4$,
proving in particular that
these strata are in $\RC$.

The techniques of the current paper can be extended to deal with
higher ranks $n=5,6$. For arbitrary rank, it is expected that Theorem
\ref{thm:motivated} holds (and hence Theorem \ref{thm:main-HC} as well),
but the proof of this probably will require a more indirect route than
the one undertaken here. Our approach has the value of sticking to
geometry for the analysis of the moduli spaces of pairs.


\section{Motives and the Hodge Conjecture}\label{sec:HC}

\subsection{Grothendieck ring of varieties} \label{subsec:K(var)}

Let $\VarC$ be the category of quasi-projective complex varieties.
We denote by $K (\VarC)$ the Grothendieck ring of
$\VarC$. This is the abelian group generated by elements $[Z]$, for
$Z \in \VarC$, subject to the relation $[Z]=[Z_1]+[Z_2]$ whenever $Z$
can be decomposed as a disjoint union $Z=Z_1\sqcup Z_2$ of a closed and
a Zariski open subset. There is a
naturally defined product in $K (\VarC)$ given by $[Y]\cdot
[Z]=[Y\x Z]$. Note that if $\pi:Z\to Y$ is an algebraic fiber
bundle with fiber $F$, which is locally trivial in the Zariski
topology, then $[Z]=[F]\cdot [Y]$.

We write $\LL:=[\AA^1]$, where $\AA^1$ is the affine line, the
\emph{Lefschetz object} in $K (\VarC)$. Clearly $\LL^k=[\AA^k]$. We
shall consider the localization $K (\VarC)[\LL^{-1}]$,
and its completion
 $$
 \hat{K} (\VarC)= \left\{ \sum_{r\geq 0} [Y_r] \, \LL^{-r} \, ; \,
 \dim Y_r -r \to -\infty \right\}.
 $$

\medskip

The ring $K (\VarC)$ has operations $\lambda^n([Y])=[\Sym^n Y]$,
extended by linearity, and satisfying the relation
$\lambda^n(a+b)=\sum_{i+j=n} \lambda^i(a)\lambda^j (b)$. 
Totaro's lemma (which appears in \cite[Lemma 4.4]{Gooo}) says that
 $$
 [\Sym^k (Y\x \AA^l)]= [\Sym^k(Y)]\x [\AA^{kl}].
 $$
Therefore the $\lambda^n$-operations can be extended to $\hat{K} (\VarC)$.

Let $C$ be a smooth projective complex curve. We define the subring
$\wRC\subset \hat{K} (\VarC)$ as the smallest set such that:
\begin{enumerate}
 \item $[C]\in \wRC$;
 \item if $a,b\in \wRC$, then $a\cdot b\in \wRC$;
 \item if $a\in \wRC$, then $\LL^k \cdot a\in \wRC$ for all $k\in \ZZ$;
 \item $\wRC$ is complete: if $a_k \in \wRC$, $\dim a_k\to -\infty$, then $\sum_{k\geq 0}a_k\in \wRC$;
 \item if $a\in \wRC$, then $\lambda^n(a)\in \wRC$ for all $n\geq 0$.
\end{enumerate}
Note that
\begin{equation}\label{eqn:wRC}
[\Jac C] \in \wRC.
\end{equation} This holds because there is a Zariski locally trivial fibration
$\PP^{g-1} \to \Sym^{2g-1}C \to \Jac C$, hence
$[\Jac C]= \lambda^{2g-1}([C]) [\PP^{g-1}]^{-1}$.

\medskip

Finally, let $\SmVarC$ denote the category of \emph{smooth projective} varieties over $\CC$.
We consider the ring $K^{bl}(\SmVarC)$ generated by the smooth projective varieties
subject to the relations $[X]-[Y]= [\Bl_Y(X)]-[E]$, where $Y\subset X$ is a smooth
subvariety, $\Bl_Y(X)$ is the blow-up of $X$ along $Y$, and $E$ is the exceptional divisor.
By \cite[Theorem 3.1]{Bit}, there is an isomorphism
 $$
  K^{bl}(\SmVarC) \cong K (\VarC).
  $$

\subsection{Motives} \label{subsec:motives}

Let us review the definition of Chow motives. A standard reference for the basic theory of classical motives, including the material presented here, is \cite{Scholl}. 
Given a smooth projective variety $X$, let $CH^d(X)$ denote the abelian group of
$\QQ$-cycles on $X$, of codimension $d$, modulo rational equivalence.
If $X,Y\in\SmVarC$, suppose that $X$ is connected and $\dim(X)=d$. The group of correspondences (of degree $0$)
from $X$ to $Y$ is $\mathrm{Corr}(X,Y)=CH^d(X\times Y)$.

For varieties $X,Y,Z\in\SmVarC$, the composition of correspondences
 $$
 \mathrm{Corr}(X,Y)\otimes\mathrm{Corr}(Y,Z)\to\mathrm{Corr}(X,Z)
 $$
is defined as
 $$
 g\circ f=p_{XZ*}(p_{XY}^*(f)\cdot p_{YZ}^*(g)),
 $$
where $p_{XZ}:X\times Y\times Z \to X\times Z$ is the projection, and similarly for $p_{XY}$ and $p_{YZ}$.

\begin{definition}
The category of \emph{(effective Chow) motives} is the category $\Mot$ such that:
\begin{itemize}
\item its objects are pairs $(X,p)$ where $X\in\SmVarC$, and
$p\in\mathrm{Corr}(X,X)$ is an idempotent ($p=p\circ p$);
\item if $(X,p),(Y,q)$ are effective motives, then the morphisms are
    $\Hom((X,p),(Y,q))=q \circ\mathrm{Corr}(X,Y)\circ p$.
\end{itemize}
\end{definition}

There is a natural functor
 \begin{equation}\label{eq:functor smvar to motives}
 h:\SmVarC^\mathrm{opp}\to\Mot
 \end{equation}
such that, for a smooth projective variety $X$,
 $$
 h(X)=(X,\Delta_X),
 $$
where $\Delta_X\in\mathrm{Corr}(X,X)$ is the graph of the identity $id_X:X\to X$.
We say that $h(X)$ is the \emph{motive of $X$}.

The category $\Mot$ is pseudo-abelian, where direct sums and tensor products are defined as follows:
  \begin{eqnarray*}
  (X,p)\oplus(Y,q) &=&(X\sqcup Y,p+q), \\
  (X,p)\otimes(Y,q)&=&(X\times Y,p_{X\times X}^*(p)\cdot p_{Y\times Y}^*(q)).
  \end{eqnarray*}
In particular
  \begin{eqnarray*}
 h(X\sqcup Y)&=&h(X)\oplus h(Y),\\
 h(X\times Y)&=&h(X)\otimes h(Y).
  \end{eqnarray*}
This allows us to define $K (\Mot)$ as the abelian group generated
by elements $[M]$, for $M\in\Mot$, subject to the relations
$[M]=[M_1]+[M_2]$, when $M=M_1\oplus M_2$. This is a ring with the product $[M_1]\cdot[M_2]=[M_1\otimes M_2]$.

\medskip

In $\Mot$, we have that $\un=h(pt)$ is the identity of the tensor product,
so it is called the \emph{unit motive}. It is easily seen that there is an isomorphism
$\un=(\PP^1,\PP^1\times pt)$. Set $\LL=(\PP^1,pt\times\PP^1)$, which is called the \emph{Lefschetz motive}
(the reason for using the same notation as for the Lefschetz object will be clear in a moment).
Therefore $h(\PP^1)=\un\oplus\LL$, and more generally,
 $$
  h(\mathbb{P}^n)=\un\oplus\LL\oplus \cdots\oplus\LL^n.
  $$

Denote also by $\LL\in K (\Mot)$ the class of the Lefschetz motive $\LL\in\Mot$.
We formally invert $\LL\in K (\Mot)$, and then consider the
completion of $K (\Mot)[\LL^{-1}]$, namely
 $$
 \hat{K} (\Mot)=\left\{\sum_{r\geq 0} [M_r] \, \LL^{-r} \, ; \,\dim M_r-r\to -\infty\right\}.
 $$


In \cite{manin} it was shown that the motive of the blow-up of a smooth projective variety $X$ along a codimension $r$ smooth subvariety $Y$ is $h(\Bl_Y(X))=h(X)\oplus\left(\bigoplus_{i=1}^{r-1}(h(Y)\otimes\LL^i)\right)$, being thus compatible with the relation defining $K^{bl}(\SmVarC)$. So the map $h$ in \eqref{eq:functor smvar to motives} descends to
$K^{bl}(\SmVarC)\to K (\Mot)$, hence defining a ring homomorphism
 \begin{equation} \label{eqn:chi}
 \chi:\hat K (\VarC)\to \hat K (\Mot).
 \end{equation}
When $X$ is smooth and projective, we have
 $$
 \chi([X])=[h(X)],
 $$
so we can think of the map $\chi$ as the natural extension of the notion of motives
to all quasi-projective varieties.
Notice that $\chi(\LL)=\LL$, which justifies the use of the same notation for the
Lefschetz object and the Lefschetz motive.

\medskip

Let $X$ be a smooth projective variety, and $F$ a finite group acting on $X$. Then, from Proposition 1.2 of \cite{B-NA}, we have
the equation
 $$
 h(X/F)=\Big( X, \frac{1}{|F|} \sum_{g\in F} \Gamma_g \Big),
 $$
where $\Gamma_g$ is the graph of $g\in F$.
In particular, $h(X/F)$ is an effective sub-motive of $h(X)$, that is
\begin{equation}\label{quotient motive}
h(X)=h(X/F)\oplus N
\end{equation} for an effective motive $N$.

From \cite{DelBano}, there are operations $\lambda^i(M)$ in $K(\Mot)$ such that
for $X$ smooth projective, we have $\lambda^i([h(X)])=[h(\Sym^i X)]$.
These operations satisfy the relation
$\lambda^n(a+b)=\sum_{i+j=n} \lambda^i(a)\lambda^j (b)$. Note that $\lambda^k(\LL^r)=\LL^{rk}$.
The map $\chi$ in (\ref{eqn:chi}) therefore commutes with the corresponding $\lambda$-operations on $K(\VarC)$ and on $K(\Mot)$.

\medskip

Let $C$ be a smooth projective curve. The motive of $C$ decomposes as
 $$
 h(C)=\un\oplus h^1(C)\oplus \LL.
 $$
It can be seen that
 $$
 h(\Jac C)=\sum_{i=0}^{2g} h^i(\Jac C),
 $$
where $[h^i(\Jac C)]=\lambda^i ([h^1(C)])$, $0\leq i \leq g$.
Finally, define
\begin{equation}\label{eqn:RC}
 \RC=\chi(\wRC)\subset \hat{K}(\Mot).
\end{equation}

\begin{definition}
We say that $X\in \VarC$ is \emph{motivated by $C$} if $\chi([X])\in \RC$.
\end{definition}

Henceforth we shall use the notation $[X]\in \hat{K}(\Mot)$ for $\chi([X])$.

\subsection{Hodge Conjecture} \label{subsec:HC}

Let $Z$ be a smooth projective variety. The Hodge structure $H^k(Z)$ has two natural
filtrations by Hodge sub-structures:
 \begin{itemize}
 \item The level filtration: $\cF^p H^k(Z)$ is the maximal Hodge sub-structure contained
 in $F^p H^k(Z)=\bigoplus_{p'\geq p\, , p'+q'=k} H^{p',q'}(Z)$.
 In particular $\cF^k H^{2k}(Z)= H^{k,k}(Z)\cap H^{2k}(Z,\QQ)$.
 \item The coniveau filtration:
  $$
  \cN^p H^k(Z)=\sum_{\codim S\geq p} \ker\left( H^k(Z)\to
  H^k(Z-S)\right)\, ,
  $$
  where $S$ runs over all Zariski closed subsets of codimension at
  least $p$. Note that $\cN^kH^{2k}(Z)$ is generated by the
  fundamental classes of algebraic cycles of codimension $k$ in
  $Z$.
 \end{itemize}
Clearly $\cN^pH^k(Z)\subset \cF^pH^k(Z)$. The (generalized) Hodge
Conjecture is the assertion $\cN^pH^k(Z) =\cF^pH^k(Z)$.
The usual Hodge Conjecture on $(k,k)$-cycles is satisfied if
$\cN^kH^{2k}(Z)= H^{k,k}(Z)\cap H^{2k}(Z,\QQ)$.

\medskip

Let $M=(X,p)$ be an effective motive. Then the cohomology of $M$ is
 \begin{equation*}
 H^k(M)=\im (p^*:H^k(X) \to H^k(X)) \, ,
 \end{equation*}
for each integer $k$. This produces a map $K(\Mot) \to K(\hs)$, $M \mapsto \sum H^k(M)$.

Let $\fhs$ be the category of filtered Hodge structures.
For $X$ smooth projective, the level
filtration $\cF^pH^k(X)$ gives an element $\cF(X)$ in $\fhs$ (see \cite{AK}).
Therefore, for $M=(X,p)$, we have an element
 $$\cF(M)=\sum \cF^\bullet H^k(M)\in K(\fhs).$$
Hence, there is a map
 $$
  \cF:K(\Mot) \to K(\fhs) \, .
 $$
Also, the coniveau filtration produces a map
 $$
  \cN: K(\Mot) \to K(\fhs)\, .
  $$
Let
 $$
 \Theta=\cF-\cN:K(\Mot) \to K(\fhs)\, .
 $$
An important remark is that for $M$ an effective motive,
$\Theta(M)$ is an effective Hodge structure.
We have the following two important properties:
  \begin{itemize}
  \item for any \emph{smooth projective} variety $X$, $\Theta (X)=0$ if and only if the
  Hodge Conjecture holds for $X$.
  \item $\Theta(\LL\cdot x)=\KK\cdot \Theta(x)$, where $\KK$ is the Hodge structure of $\AA^1$ (the one-dimensional
  Hodge structure with weight $(1,1)$).
  \end{itemize}
Unfortunately, the map $\Theta$ is not known to be a multiplicative
map. However, the
important feature is that the kernel of $\Theta$ is an abelian subgroup
stable by multiplication by $\LL$ and $\LL^{-1}$, which allows us to
consider the map on the completions
\begin{equation}\label{eqn:Theta}
 \Theta: \hat K(\Mot) \to \hat K(\fhs)\, .
\end{equation}

\begin{lemma}\label{lem:HC sum}
If $M$ and $N$ are effective motives, then
 $$
 \Theta(M\oplus N)=0 \Longleftrightarrow \Theta(M)=\Theta(N)=0\, .
 $$
\end{lemma}

\begin{proof}
 In the left direction, it is obvious. In the right direction, just note
 that $\Theta$ sends effective motives to effective filtered Hodge structures.
\end{proof}

For $X\in \VarC$, we write $\Theta(X):=\Theta(h(X))$. 

The following result is well-known (cf. \cite{A}). For convenience of the reader, we include a proof. 
The genericity of the curve appearing in the proposition means that it belongs to the complement of a 
countable union of some closed sets of the moduli of curves (see \cite[Section 1]{BP} or \cite[Proposition 6.5]{A}). 
\begin{proposition} \label{prop:generic}
  Let $C$ be a generic curve. Then
  \begin{itemize}
  \item $\Theta(C^k)=0$, for any $k\geq 0$.
  \item $\Theta(\Sym^k C)=0$, for any $k\geq 0$.
  \item $\Theta(\Jac C)=0$.
  \item $\Theta(\RC)=0$.
  \end{itemize}
\end{proposition}

\begin{proof}
  Let $A=\Jac C$, which is a polarised abelian variety. The Hodge
  group of $A$, $\Hg(A)$, is defined as the group of all linear
  automorphisms of $V=H^1(A,\QQ)$ which leave invariant all Hodge
  cycles of the varieties $A\x \cdots \x A$. Let $E$ denote the
  polarisation of $A$. Then $\Hg(A)=\Sp(V,E)$. This can be proved
  with the arguments of \cite{BP}: the Hodge group of $\Jac C$ for
  general $C$ contains the Hodge group of any degeneration of $C$,
  in particular, degenerating $C$ to a reducible nodal curve
  consisting of the union of two curves of genus $a$ and $g-a$, we
  see that $\Sp(2g)\supset \Hg(A)\supset \Sp(2a)\x \Sp(2g-2a)$. As
  the general Jacobian has Neron-Severi group equal to $\ZZ$, the
  argument in \cite[Theorem 5]{BP} proves that $\Hg(A)=\Sp(2g)$.

By \cite[Proposition 17.3.4]{BL}, $\End_\QQ(A)=\End(V)^{\Hg(A)}=\QQ$.
 Now the Lefschetz group of $A$ is
 $$
 Lf(A)=\{g \in \Sp(V,E) \ ;\  g\circ f=f\circ g, \forall f\in  \End_\QQ(A) \}_0
 =\Sp(V,E)=\Hg(A)\, .
 $$
By \cite{Tankeev} or \cite[Exercise 13]{BL},  $H^*_{Hodge}(A^k)$
is generated by divisors, so the Hodge Conjecture holds for $A^k$, for all
$k\geq 1$. This is written, in our terminology, as $\Theta((\Jac C)^k)=0$.

Now $h^1(C)$ is a submotive of $h(\Jac C)$, so Lemma \ref{lem:HC sum} implies that
$\Theta(h^1(C)^k)=0$, and therefore $\Theta(C^k)=0$ for any $k\geq 1$.
From this, using \eqref{quotient motive} and Lemma \ref{lem:HC sum}, we get $\Theta(\Sym^k C)=0$, for any $k\geq 1$.
Finally, $\RC$ is generated by varieties obtained by
taking iterated products and symmetric products of the
curve $C$. Such a manifold is isomorphic to a quotient $C^k/F$, where $k\geq 1$ and
$F\subset \frS_k$ is a subgroup of
the permutation group of the factors of $C^k$. Using \eqref{quotient motive} and
Lemma \ref{lem:HC sum} again, it follows that $\Theta(C^k/F)=0$.
\end{proof}

\section{Moduli spaces of pairs and of triples} \label{sec:triples}

\subsection{Moduli spaces of pairs}

Let $C$ be a smooth projective curve of genus $g\geq 3$ over $\CC$.
We denote by $M(n,d)$ the moduli space of polystable bundles of rank
 $n$ and degree $d$ over $C$. The open subset consisting of stable
 bundles will be denoted $M^s(n,d) \subset M(n,d)$. Note that
 $M(n,d)$ is a projective variety, which is in general not smooth if
 $n$ and $d$ are not coprime. On the other hand, $M^s(n,d)$ is a
 smooth quasi-projective variety. If $L_0$ is a fixed line bundle of
 degree $d$, then we have the moduli spaces $M^s(n,L_0)$ and $M(n,L_0)$
 consisting of stable and polystable bundles $E$, respectively, with
 determinant $\det(E) = L_0$.

A pair $(E,\phi)$ over $C$ consists of a vector bundle $E$ of rank
$n$ and degree $d$, and $\phi\in H^0(E)$. Let $\tau\in \RR$. We say that
$(E,\phi)$ is $\tau$-stable (see \cite[Definition 4.7]{GP}) if:
 \begin{itemize}
 \item For any subbundle $E'\subset E$, we have $\mu(E')<\tau$.
 \item For any subbundle $E'\subset E$ with $\phi\in H^0(E')$, we
 have $\mu(E/E')>\tau$.
 \end{itemize}
The concept of $\tau$-semistability is defined by replacing the
strict inequalities by weak inequalities. A pair $(E,\phi)$ is
$\tau$-polystable if $E=E'\oplus E''$, where $\phi\in H^0(E')$, $(E',\phi')$
is $\tau$-stable, and
$E''$ is a polystable bundle of slope $\tau$. The moduli space of
$\tau$-polystable pairs is denoted by $\frM_\tau ( n, d)$.
It is a projective variety and contains a smooth open subset
$\frM_\tau^s(n,d)\subset
 \frM_\tau(n,d)$ consisting of $\tau$-stable pairs.

 If we fix the
 determinant $\det(E)= L_0$, then we have the moduli spaces of
 pairs with fixed determinant, $\frM_\tau^s(n,L_0)$ and
 $\frM_\tau(n,L_0)$. Pairs are discussed at length in
 \cite{B,BD,GP,Mu1,Mu2,Mu-HS,MOV1}.

 The range of the parameter $\tau$ is an open interval $I$ split by a
 finite number of critical values $\tau_c$. For a non-critical value
 $\tau\in I$, there are no properly polystable pairs, so $\frM_\tau
 (n,d)=\frM_\tau^s(n,d)$ is smooth and projective. For a critical
 value $\tau=\tau_c$, $\frM_\tau(n,d)$ is in general singular at
 properly $\tau$-polystable points.

\subsection{Moduli spaces of triples}

As it has been done in other articles \cite{Mu1,Mu2,Mu-HS,MOV1}, it is convenient
to rephrase questions about pairs into a more general object known as a triple \cite{BGP,BGPG}.

A triple $T = (E_{1},E_{2},\phi)$ on $C$ consists of two vector
bundles $E_{1}$ and $E_{2}$ over $C$, of ranks $n_1$ and $n_2$ and
degrees $d_1$ and $d_2$, respectively, and a homomorphism $\phi
\colon E_{2} \to E_{1}$. We shall refer to $(n_1,n_2,d_1,d_2)$ as
the {type} of the triple. 
For any $\s \in \RR$, the $\s$-slope of $T$ is defined by
 $$
   \mu_{\s}(T)  =
   \frac{d_1+d_2}{n_1+n_2} + \s \frac{n_{2}}{n_{1}+n_{2}}\ .
 $$
We say that a triple $T = (E_{1},E_{2},\phi)$ is $\s$-stable if
$\mu_{\s}(T') < \mu_{\s}(T)$ for any proper subtriple $T' =
(E_{1}',E_{2}',\phi')$. We define $\s$-semistability by replacing
the above strict inequality with a weak inequality. A triple $T$ is
$\s$-polystable if it is the direct sum of $\s$-stable triples of
the same $\s$-slope. We denote by
  $$
  \cN_\s(n_1,n_2,d_1,d_2)
  $$
the moduli space of $\s$-polystable triples of type
$(n_1,n_2,d_1,d_2)$. This moduli space was constructed in \cite{BGP}
and \cite{Sch}. It is a complex projective variety. The open subset
of $\s$-stable triples will be denoted by
$\cN_\s^s(n_1,n_2,d_1,d_2)$.

Let $L_1,L_2$ be two bundles of degrees $d_1,d_2$ respectively. Then
the moduli spaces of $\sigma$-polystable triples $T=(E_1,E_2,\phi)$
with $\det(E_1)=L_1$ and $\det(E_2)=L_2$ will be denoted by
  $$
  \cN_\s(n_1,n_2,L_1,L_2)\, ,
  $$
and $\cN_\s^s(n_1,n_2,L_1,L_2)$ is the open subset of
$\s$-stable triples.

\medskip

For the case $(n_1,n_2)=(n,1)$, we recover the notion of a pair.
Given a pair $(E,\phi)$, we interpret
$\phi\in H^0(E)$ as a morphism $\phi:\cO \to E$, where
$\cO$ is the trivial line bundle on $X$. So we have an identification
$(E,\phi)\mapsto (E,\cO,\phi)$ from pairs to triples. The
$\tau$-stability of $(E,\phi)$ corresponds to the $\s$-stability of
$(E,\cO,\phi)$, where $\s=(n+1)\tau -d$ (see \cite{BGP}). Therefore
we have an isomorphism of moduli spaces
   \begin{equation}\label{eqn:isom}
    \cN_\sigma (n,1,d,0)  \cong \frM_\tau (n,d) \x \Jac C\, ,
   \end{equation}
given by $(E,L,\phi)\mapsto ((E\ox L^*,\phi),L)$.
In the case of fixed determinant, we have
 $$
 \cN_\sigma (n,1,L_0,\cO)  \cong \frM_\tau (n,L_0).
 $$
We shall take the point of view of triples for studying $\frM_\tau(n,d)$, because
triples are
more adapted to using homological algebra (extensions, filtrations, etc).
Henceforth, we shall write
 $$
 \cN_\s:=\cN_\s(n,1,d,d_o).
 $$

\subsection{Critical values and flip loci}

Let $\mu(E)=\deg(E)/\rk(E)$ denote the slope of a bundle $E$, and
let $\mu_i=\mu(E_i)=d_i/n_i$, for $i=1,2$. Write
  \begin{align*}
  \s_m = &\, \mu_1-\mu_2\ ,  \\
  \s_M = & \left\{ \begin{array}{ll}
    \left(1+ \frac{n_1+n_2}{|n_1 - n_2|}\right)(\mu_1 - \mu_2)\ ,
      \qquad & \mbox{if $n_1\neq n_2$\ ,} \\ \infty, & \mbox{if $n_1=n_2$\
      ,}
      \end{array} \right.
  \end{align*}
and let $I$ be the interval $I=(\s_m,\s_M)$. Then a necessary
condition for $\cN_\s^s(n_1,n_2,d_1,d_2)$ to be non-empty is that
$\s\in I$ (see \cite{BGPG}).
To study the
dependence of the moduli spaces on the parameter $\s$, we need
the concept of critical value \cite{BGP,MOV1}.

\begin{definition}\label{def:critical}
The values $\s_c\in I$ for which there exist $0 \le n'_1 \leq
n_1$, $0 \le n'_2 \leq n_2$, $d'_1$ and $d'_2$, with $n_1'n_2\neq
n_1n_2'$, such that
$$
 \s_c=\frac{(n_1+n_2)(d_1'+d_2')-(n_1'+n_2')(d_1+d_2)}{n_1'n_2-n_1n_2'},
$$
are called \emph{critical values}. We also consider $\s_m$ and $\s_M$
(when $\s_M\neq \infty$) as critical values.
\end{definition}

\begin{theorem}[\cite{BGPG}]\label{thm:pairs}
 For non-critical values $\s\in I$, $\cN_{\s}$ is smooth and
 projective, and it consists only of $\s$-stable points (i.e.
 $\cN_{\s}=\cN_{\s}^s$). For critical values $\s=\s_c$,
 $\cN_{\s}$ is projective, and the open subset
 $\cN_{\s}^s\subset \cN_{\s}$ is smooth. In both cases, the dimension of
 $\cN_{\s}$ is $(n^2-n+1)(g-1)+1+d-n \, d_o$.
\end{theorem}

The interval $I$ is split by a finite number of values $\s_c \in I$.
The stability and semistability criteria  for two values of $\s$
lying between two consecutive critical values are equivalent; thus
the corresponding moduli spaces are isomorphic. When $\s$ crosses a
critical value, the moduli space undergoes a transformation which we
call a \emph{flip}. Let $\s_c\in
I$ be a critical value and set $\scp = \s_c + \epsilon$, $\scm = \s_c -
 \epsilon$,
where $\epsilon > 0$ is small enough so that $\s_c$ is the only
critical value in the interval $(\scm,\scp)$.
We define the \textit{flip loci} as
 \begin{align*}
 \cS_{\scp} &= \{ T\in\cN_{\scp}^s \ ;
 \ \text{$T$ is $\scm$-unstable}\} \subset\cN_{\scp}^s \ ,\\
 \cS_{\scm} &= \{ T\in\cN_{\scm}^s \ ;
 \ \text{$T$ is $\scp$-unstable}\}
 \subset\cN_{\scm}^s \ .
 \end{align*}

It follows that (see \cite[Lemma 5.3]{BGPG})
\begin{equation}\label{eqn:flip}
 \cN_{\scp}^s\setminus\cS_{\scp}=\cN_{\s_c}^s=\cN_{\scm}^s\setminus\cS_{\scm}.
\end{equation}

\medskip

When $d/n-d_o> 2g-2$, the moduli space $\cN_{\s}$ for the smallest possible values of the
parameter can be described explicitly. Let $\smp=\s_m+\epsilon$,
$\epsilon>0$ small enough. By \cite[Proposition 4.10]{MOV1}, there
is a morphism
 \begin{equation}\label{eqn:alfa}
 \pi:\cN_{\smp} =\cN_{\smp}(n,1,d,d_o) \to M (n,d) \x \Jac C
 \end{equation}
which sends $T=(E,L,\phi)\mapsto (E,L)$. Let
\begin{equation}\label{eqn:Um}
  \cU_m=\cU_m(n,1,d,d_o)= \pi^{-1}(M^s(n,d)\x \Jac C).
\end{equation}
By \cite[Proposition 4.10]{MOV1},
$\pi:\cU_m \to M^s(n,d) \x \Jac C$ is a projective fibration whose fibers are
the projective spaces $\PP H^0(E\ox L^*)$. We write
\begin{equation}\label{eqn:Dm}
 \cD_m:=\cD_m(n,1,d,d_o)= \cN_{\smp}\setminus \cU_m \, .
\end{equation}

\subsection{The flip locus $\cS_{\s_c^+}$}\label{sec:flipS+}

We start by describing geometrically $\cS_{\s_c^+}$. The
following description is taken from \cite{Mu-HS}.
Let $b\geq 1$ and
$r\geq 1$. Fix $n'\geq 1$ and $d'$ such that
 \begin{equation}\label{eqn:1}
 \frac{d'+d_o+\s_c}{n'+1} = \mu_{\s_c}(T) = : \mu_c \,  .
 \end{equation}
Let $(n_1,d_1),\ldots, (n_b,d_b)$ satisfy
 \begin{equation}\label{eqn:2}
 \frac{d_i}{n_i}=\mu_c\,.
 \end{equation}
Consider $a_{ij}\geq 0$, for $1\leq i\leq b$ and $1\leq j\leq r$,
such that $\ba_j=(a_{1j},\ldots,a_{bj})\neq (0,\ldots,0)$, for all $j$.
We assume that
 \begin{equation}\label{eqn:3}
 \sum_{i,j} a_{ij} n_i + n'= n.
 \end{equation}
Write
$\bn=((n_i), (a_{i1}), \ldots, (a_{ir}))$, which
we call the \emph{type} of the stratum.
Consider
 \begin{equation}\label{eqn:3.5}
 \tilde U(\bn)= \{(E_1,\ldots, E_b)\in M^s(n_1,d_1)\times \cdots \times
 M^s(n_b,d_b) \, ; \, E_i\not\cong E_j, \ \text{for } i\neq j \}\,.
 \end{equation}
For each $(E_1,\ldots, E_b)\in \tilde U(\bn)$, set
$S_i=(E_i,0,0)$, $1\leq i\leq b$.

We define $X^+(\bn)\subset \cS_{\s_c^+}$ as the subset formed by
those triples $T$ admitting a filtration
 $$
0=T_0\subset
T_1\subset T_2\subset \cdots \subset T_{r+1}=T
$$
such that, for some $(E_1,\ldots, E_b)\in \tilde U(\bn)$,
$$
  \bar{T}_j=T_j/T_{j-1} \cong S(\ba_j):= S_1^{a_{1j}} \oplus \cdots  \oplus
  S_b^{a_{bj}}\, ,
$$
is the maximal
$\s_c$-polystable subtriple of $T/T_{j-1}$. Note that $\bar T_{r+1}\in \cN_{\s_c}^s(n',1,d',d_o)$. By \cite[Lemma 4.8]{Mu-HS},
  $$
  \cS_{\s_c^+}=\bigsqcup_{\bn} X^+(\bn)\, .
  $$

\begin{proposition}[{\cite[Proposition 5.1]{Mu-HS}}] \label{prop:Xbn}
  Let $(E_1,\ldots, E_b, \bar T_{r+1})\in \cM(\bn):=\tilde U(\bn) \times \cN_{\s_c}^s(n',1,d',d_o)$.
  Define triples $\bT_j$ by downward recursion as follows:
  $\bT_{r+1}=\bar T_{r+1}$ and for $1\leq j\leq r$
  define $\bT_j$ as an extension
   \begin{equation}\label{eqn:ext}
   0\to S(\ba_j)\to \bT_j \to \bT_{j+1}\to 0\, .
   \end{equation}
  Let $\xi_j\in \Ext^1(\bT_{j+1},S(\ba_j))$ be the extension class corresponding to (\ref{eqn:ext}).
  Write $T:=\bT_1$. Then $T\in X^+(\bn)$ if and only if the
  following conditions are satisfied:
  \begin{enumerate}
  \item The extension class $\xi_j\in \Ext^1(\bT_{j+1},S(\ba_j))= \prod_i
  \Ext^1(\bT_{j+1},S_i)^{a_{ij}}$ lives in \newline
  $\prod_i V(a_{ij},\Ext^1(\bT_{j+1},S_i))$, with the notation \newline
  $V(k,W)=\{(w_1,\ldots, w_k)\in W^k \ ; \ w_1,\ldots, w_k \text{
  are linearly independent}\}$, \newline for $W$ a vector space.
  \item Consider the
   map $S(\ba_{j+1})\to \bT_{j+1}$ and the element $\xi_j'$ which
   is the image of $\xi_j$ under
   $\Ext^1(\bT_{j+1},S(\ba_j)) \to \Ext^1(S(\ba_{j+1}),S(\ba_j))$.
   Then the class $\xi_j'\in \Ext^1(S(\ba_{j+1}),S(\ba_j))= \prod_i
  \Ext^1(S_i,S(\ba_{j}))^{a_{i,j+1}}$ lives in $\prod_i
  V(a_{i,j+1},\Ext^1(S_i,S(\ba_{j})))$.
  \end{enumerate}

  Two extensions $\xi_j$ give rise to isomorphic $\bT_j$
  if and only if the triples $\bT_{j+1}$ are
  isomorphic and the extension classes are the same up to action
  of the group $\GL(\ba_j):=\GL(a_{1j})\times \cdots \times \GL(a_{bj})$.
\end{proposition}

There is a fiber bundle
 $$
 \tilde{\tilde{X}}^+(\bn)\to \cM(\bn)
 $$
whose fiber $F$ consists of the iterated extensions satisfying the
conditions in Proposition \ref{prop:Xbn}. Therefore
 $$
 \tilde{X}^+(\bn):= \tilde{\tilde{X}}^+(\bn)/\GL(\bn) \to \cM(\bn),
 $$
where $\GL(\bn)=\prod \GL(\ba_j)$, has fiber $F/\GL(\bn)$.
The finite group
   \begin{equation} \label{eqn:F}
 \frS_\bn=\{\tau \text{ permutation of } (1,\ldots, b) \, ; \,
   n_{\tau(i)}=n_i, a_{\tau(i)j}=a_{ij} \ \forall i,j \ \}
  \end{equation}
acts (freely) on $\tilde U(\bn)$ and on $\tilde X^+(\bn)$, by permuting the bundles.
The quotient is the fibration (locally trivial in the usual topology)
 $$
 X^+(\bn):=\tilde{X}^+(\bn)/\frS_\bn \to {U}(\bn)=\tilde U(\bn)/\frS_\bn \, .
 $$

\medskip

 The description in Proposition \ref{prop:Xbn} can be applied to the
 critical value $\s_c=\s_m$. In this case,
 $\cN_{\s_m^+}=\cS_{\s_m^+}$. The only difference is that now
 $\bT_{r+1}$ should be of the form $L\to 0$, that is,
 $\bT_{r+1}\in \cN_{\s_c}^s(0,1,0,d_o)=\Jac^{d_o}X$.
 Note that there is an open stratum in $\cS_{\s_m^+}$ corresponding
 to $r=1$, $b=1$, $\bn_o=((n),(1))$. In this case
  $$
  \cM(\bn_o)=M^s(n,d)\times \Jac^{d_o} X.
  $$
 This corresponds to triples
 $\phi:L\to E$ for which $E$ is a stable bundle. So the stratum $X^+(\bn_0)$
 is equal to $\cU_m\subset \cN_{\s_m^+}$, defined in \eqref{eqn:Um}. The remaining strata
 compose $\cD_m$, and all have $n_i<n$.

\subsection{The flip locus $\cS_{\s_c^-}$}

There is an analogous description for $\cS_{\s_c^-}$. As before,
let $b\geq 1$ and $r\geq 1$. Fix $n'\geq 1$ and $d'$ satisfying
(\ref{eqn:1}). Let $(n_1,d_1),\ldots, (n_b,d_b)$ satisfy
(\ref{eqn:2}). Consider $a_{ij}\geq 0$, for $1\leq i\leq b$ and
$2\leq j\leq r+1$, such that $\ba_j=(a_{1j},\ldots,a_{bj})\neq
(0,\ldots,0)$, for all $j$. We assume (\ref{eqn:3}). Write
$\bn=((n_i), \ba_1, \ldots, \ba_r)$.
We consider $\tilde U(\bn)$ as in (\ref{eqn:3.5}).

We define $X^-(\bn)\subset \cS_{\s_c^-}$ as the subset formed by
those triples $T$ admitting a filtration $0=T_0\subset
T_1\subset T_2\subset \cdots \subset T_{r+1}=T$, such that
$$
  \bar{T}_j=T_j/T_{j-1} \cong S(\ba_j):= S_1^{a_{1j}} \oplus \cdots  \oplus
  S_b^{a_{bj}}\, ,
$$
is the maximal polystable subtriple of $T/T_{j-1}$,
where $(E_1,\ldots, E_b)\in \tilde U(\bn)$, $2\leq j\leq r+1$. It must
be $T_{1}\in \cN_{\s_c}^s(n',1,d',d_o)$. Then
  $$
  \cS_{\s_c^-}=\bigsqcup_{\bn} X^-(\bn)\, .
  $$

Note again that the finite group $\frS_\bn$ given in (\ref{eqn:F}) acts
on $\tilde U(\bn)$. Then there is a fibration
 \begin{equation*}
 \tilde{X}^-(\bn)\to \cM(\bn):= \tilde U(\bn) \times \cN_{\s_c}^s(n',1,d',d_o).
  \end{equation*}
and $X^-(\bn)=\tilde{X}^-(\bn)/ \frS_\bn$.

\begin{proposition}[{\cite[Proposition 5.3]{Mu-HS}}] \label{prop:Xbn-}
  Let $(E_1,\ldots, E_b, T_{1})\in \cM(\bn)$.
  Define triples $\bT_j$ by recursion as follows:
  $\bT_{1}=T_{1}$, and for $2\leq j\leq r+1$
  define $\bT_j$  as an extension
   $$
   0\to \bT_{j-1} \to \bT_{j}\to S(\ba_j)\to 0\, .
   $$
  Let $\xi_j\in \Ext^1(S(\ba_j),\bT_{j-1})$ be the corresponding extension class.
  Write $T:=\bT_{r+1}$. Then $T\in X^-(\bn)$ if and only if the
  following conditions are satisfied:
  \begin{enumerate}
  \item The extension class $\xi_j\in \Ext^1(S(\ba_j),\bT_{j-1})= \prod_i
  \Ext^1(S_i,\bT_{j-1})^{a_{ij}}$ lives in \newline
  $\prod_i  V(a_{ij}, \Ext^1(S_i,\bT_{j-1}))$.
   \item Consider the
   map $\bT_{j-1}\to S(\ba_{j-1})$ and the element $\xi_j'$ which
   is the image of $\xi_j$ under
   $\Ext^1(S(\ba_j),\bT_{j-1}) \to \Ext^1(S(\ba_{j}),S(\ba_{j-1}))$.
   Then the element $\xi_j'\in \Ext^1(S(\ba_{j}),S(\ba_{j-1}))= \prod_i
  \Ext^1(S(\ba_{j}),S_i)^{a_{i,j-1}}$ lives in $\prod_i
  V(a_{i,j-1}, \Ext^1(S(\ba_{j}),S_i))$.
  \end{enumerate}

  Two extensions $\xi_j$ give rise to isomorphic $\bT_{j}$
  if and only if the triples $\bT_{j-1}$ are
  isomorphic and the extension classes are the same up to action
  of the group $\GL(\ba_j):=\GL(a_{1j})\times \cdots \times \GL(a_{bj})$.
\end{proposition}

 The description in Proposition \ref{prop:Xbn-} is also valid for $\s_c=\s_M$, with no change.
 In this case, $\cN_{\s_M^-}=\cS_{\s_M^-}$.

\section{The strata $X^+(\bn)$ for $r=1$} \label{sec:r=1}

Now we move on to the issue of giving an explicit description
for the strata $X^\pm (\bn)$ corresponding to a critical value
$\sigma_c$ and a type $\bn=((n_i), \ba_1, \ldots, \ba_r)$, using Propositions \ref{prop:Xbn}
and \ref{prop:Xbn-}. Recall that each $n_i$ determines the corresponding $d_i$ by \eqref{eqn:2}. Our aim is to prove that $[X^\pm(\bn)]\in
\RC$.

We start with a simple case.

\begin{proposition}\label{prop:no-finite-group}
Let $\sigma_c$ be any critical value (possibly $\sigma_c=\sigma_m,\sigma_M$).
Let $n\geq 1$, and $\bn=((n_i), \ba_1, \ldots, \ba_r)$ be a type
such that $\frS_{\bn}=\{1\}$. If $\sigma_c=\sigma_m$, we assume $\bn\neq \bn_0$. Suppose that $[M^s(n'',d'')]$ and $[\cN_{\sigma_c}^s(n',1,d',d_o)]$ are in $\RC$, for any $n',n''<n$,
$\gcd(n'',d'')=1$. Assume also that $\gcd(n_i,d_i)=1$, for every $i=1,\ldots,b$.
Then $[X^\pm(\bn)]\in K(\Mot)$ belongs to $\RC$.
\end{proposition}

\begin{proof}
If the group $\frS_\bn$ is trivial, then by \cite[Theorem 6.1]{Mu-HS}, the fibration
 $$
 F \to  \tilde{\tilde{X}}^\pm(\bn) \to \cM(\bn)
 $$
has fiber $F$ which is \emph{affinely stratified} (AS, for short). That means
that it is an iterated fiber bundle of spaces which are
an affine space minus a linear subspace. It is easy to see that a Zariski
locally trivial fibration whose base and fibre are both AS has AS total space.
Therefore, $[F]=P(\LL)$ for some polynomial $P$.
Then
 \begin{equation}\label{zzz}
 F/\GL(\bn) \to X^\pm(\bn) \to \cM(\bn)
 \end{equation}
has fiber such that $[F/\GL(\bn)]=P(\LL)/[\GL(\bn)]$, which is
a series in $\LL$.

As $\gcd(n_i,d_i)=1$, there are universal bundles $\cE_i  \to M^s(n_i,d_i) \x C$.
Analogously, there is a universal triple $\cT' \to \cN' \x C$
since there are universal bundles over any moduli space
of $\sigma$-stable triples of type $(n,1)$.
The bundle (\ref{zzz}) is constructed iteratively by taking bundles of
relative $\Ext^1$-groups of the $\cE_i$'s and $\cT'$. All these bundles are 
then Zariski locally trivial, since the dimension of the $\Ext^1$-groups is 
constant, because $\Ext^0=\Ext^2=0$ in all these cases, as it is proven
in \cite{Mu-HS}.
This means that 
(\ref{zzz}) is locally trivial in the Zariski topology, and hence
$[X^\pm(\bn)]= [\cM(\bn)]\, [F/\GL(\bn)]$.

The basis of the fibration is the space
 $$
 \cM(\bn)= \tilde U(\bn)\times \cN_{\s_c}^s(n',1,d',d_o) =
 \left( \prod_{i=1}^b M^s(n_i,d_i) \setminus \Delta \right) \times \cN_{\s_c}^s(n',1,d',d_o)\, ,
 $$
where $\Delta$ is a union of some diagonals, each of which is a product of some moduli
spaces $M^s(n_j,d_j)$.
Therefore
 $$
 [X^\pm(\bn)]= [\tilde U(\bn)]\, [\cN_{\s_c}^s(n',1,d',d_o)]\, P(\LL)/[\GL(\bn)]
 \in \RC.
 $$
\end{proof}

Now we consider the strata $X^+(\bn)\subset\cS_{\s_c^+}$ where the standard filtration
has only one nontrivial step, hence it is of the form $0\subset T_1\subset T_2=T$.
So $r=1$, and we only have $\ba_1=(a_1, a_2, \ldots, a_b)$,
where all $a_i>0$. We write the type as a matrix
\begin{equation}\label{eqn:bn}
\bn= \begin{pmatrix} (n_i) \\ (a_{i1}) \end{pmatrix} =
\begin{pmatrix}
    n_1 & n_2 & \cdots & n_b \\
    a_1 & a_2 & \cdots & a_b
\end{pmatrix}.
\end{equation}

To describe geometrically the strata $X^+(\bn)$ in this and in
the following section, we will make use of  partitions of sets. If
  $$[b]:=\{1,\ldots,b\}$$
denotes the set of the first $b$ positive integer numbers, we define a \emph{partition} $\pi$ to
be a collection of disjoint subsets of $[b]$ whose union is $[b]$. An
element of a partition is called a \emph{brick}. We can define a partition by an
equivalence relation, for which the bricks are the equivalence classes.

There is a natural partial order on the set of partitions of $[b]$.
If $\pi,\pi'$ are two such partitions, then we say that $\pi\leq \pi'$
if any $\beta\in \pi$ is a subset of some $\beta'\in \pi'$.

Let $\pi'$ and $\pi''$ be two
partitions on $[b]$. Then the \emph{intersection} $\pi'\wedge \pi''$ of
the partitions $\pi'$ and $\pi''$ is defined as
\[
\pi'\wedge \pi'' =\{ \beta'\cap \beta''\subset [b]\, ; \, \beta'\in \pi'\text{ and
}\beta''\in\pi''\}.
\]
Obviously, $\pi'\wedge\pi''\leq \pi'$ and $\pi'\wedge\pi''\leq \pi''$.

\begin{proposition}\label{prop:r=1}
Let $\sigma_c$ be any critical value.
Let $n\geq 1$, and $\bn$ be given by \eqref{eqn:bn},
and assume that $\frS_\bn$ is non-trivial. Suppose that
$[M^s(n'',d'')]$ and $[\cN_{\sigma_c}^s(n',1,d',d_o)]$ are in $\RC$, for any $n',n''<n$,
$\gcd(n'',d'')=1$. Assume also that $\gcd(n_i,d_i)=1$, for every $i=1,\ldots,b$. Then $[X^\pm(\bn)]\in \RC$. 
\end{proposition}

\begin{proof}
We will only do the case of $X^+(\bn)$, the other one being analogous.
Given $\bn$, we have two partitions of $[b]$: $\pi_0$ defined by the equivalence relation
\begin{equation}\label{eq:pi0partition}
i\sim j \iff n_i=n_j,
\end{equation}
and $\pi_1$ defined by
 $$i\sim j\iff n_i=n_j\text{ and }a_i=a_j.$$
Of course, we have $\pi_1\leq \pi_0$.
With the notations of Section \ref{sec:flipS+},
 \[
 \tilde{U}(\bn)=\prod_{\beta\in \pi_0} \left( M^s(n_\beta,d_\beta)^{|\beta|} \setminus \Delta_\beta\right),
 \]
where $\Delta_\beta$ stands for the big diagonal, and
we write $n_\beta:=n_i$ and $d_\beta:=d_i$ for any $i\in \beta$.
The space $\tilde{X}^+(\bn)$ is a bundle over $\tilde{U}(\bn)\times
\cN'$, where $\cN'=\cN_{\s_c}^s(n',1,d',d_o)$, whose fiber over $(S_1,\ldots,S_b,T')$ is
 \begin{equation}\label{eqn:hoy1}
 F =\prod_{\beta\in \pi_0}\prod_{\substack{\gamma\in \pi_1\\ \gamma\subset
 \beta}}\Gr(a_\gamma,\Ext^1(T',S_\gamma))^{|\gamma|} = \prod_{\gamma\in \pi_1}
 \Gr(a_\gamma,\Ext^1(T',S_\gamma))^{|\gamma|} .
 \end{equation}

\medskip \noindent \textbf{Step 1.} We extend this bundle to a
bundle $\tilde{X}^+(\bn)^\Delta$ over the union
of all diagonals, that is, we have the bundle
\[
\tilde{X}^+(\bn)^\Delta\to \prod_{\beta\in \pi_0}
M^s(n_\beta,d_\beta)^{|\beta|} \times\cN'= \prod_{\gamma\in \pi_1}
M^s(n_\gamma,d_\gamma)^{|\gamma|}\times\cN' ,
\]
with fiber (\ref{eqn:hoy1}).

We want to write the total space $\tilde{X}^+(\bn)^\Delta$ as a product of
bundles. For each $\gamma\in\pi_1$, 
we have a Zariski locally trivial bundle
 \begin{equation}\label{eqn:factor bundle}
 \cE_\gamma \to M^s(n_\gamma,d_\gamma)\x \cN'
 \end{equation}
with fiber
$\Gr(a_\gamma,\Ext^1(T',S_\gamma))$ over $(S_\gamma,T')\in M^s(n_\gamma,d_\gamma)\x \cN'$.
We can consider $\cE_\gamma$ as a bundle over $\cN'$.
Its fiber over $T'$ is the Zariski locally trivial bundle $\cE_{\gamma,T'}$ with basis
$M^s(n_\gamma,d_\gamma)$ and fiber 
$\Gr(a_\gamma,\Ext^1(T',S_\gamma))$.
Then $\tilde{X}^+(\bn)^\Delta$ is the fiber product
of the bundles $(\cE_\gamma)^{|\gamma|}$ over $\cN'$:
 $$
 \tilde{X}^+(\bn)^\Delta= \prod_{\substack{\gamma\in \pi_1 \\ \cN'}} (\cE_{\gamma})^{|\gamma|}.
 $$
Here we are thinking of $\tilde{X}^+(\bn)^\Delta$ as a bundle over $\cN'$, as we just did
for each $\cE_\gamma$.

Let us take the quotient of $\tilde{X}^+(\bn)^\Delta$ by the finite group $\frS_\bn$.
We use partitions to describe these groups. For any partition $\pi$ of $[b]$, let
$$\frS_\pi=\prod_{\beta\in \pi} \frS_\beta.$$
So $\frS_\bn=\frS_{\pi_1}$. Then
 $$
 \tilde{X}^+(\bn)^\Delta/\frS_\bn= \prod_{\substack{\gamma\in \pi_1 \\ \cN'}} (\cE_{\gamma})^{|\gamma|}/\frS_\gamma
 = \prod_{\substack{\gamma\in \pi_1 \\ \cN'}} \Sym^{|\gamma|} (\cE_{\gamma})\, ,
 $$
where $\Sym^{|\gamma|} (\cE_{\gamma})$ is the bundle over $\cN'$ whose fiber over $T'$ is
the symmetric product $\Sym^{|\gamma|} (\cE_{\gamma,T'})$.
Therefore, since by assumption $[M^s(n_\gamma,d_\gamma)], [\cN']\in\RC$, it follows that $\cE_{\gamma,T'}$ is motivated by $C$ and, by definition, the same holds for $\Sym^{|\gamma|} (\cE_{\gamma,T'})$. So, we have
 $$
 [\tilde{X}^+(\bn)^\Delta/\frS_\bn] \in \RC\, .
 $$

\medskip \noindent \textbf{Step 2.}
Now we deal with the diagonals. Consider the partition $\pi_0$ defined by \eqref{eq:pi0partition}, and let $\beta$ be a subset of a brick of $\pi_0$. We define
$$
 \cE_\beta^\Delta = \prod_{\substack{i\in \beta \\ M^s(n_\beta,d_\beta)\times \cN'}} \cE_i\, ,
$$
as the fiber product over $M^s(n_\beta,d_\beta) \times \cN'$ of the bundles $\cE_i \to M^s(n_\beta,d_\beta) \times \cN'$ given in \eqref{eqn:factor bundle}. So there is a
fibration
$$
  \prod_{\substack{\gamma\in \pi\wedge \pi_1\\ \gamma\subset
 \beta}}\Gr(a_\gamma,\Ext^1(T',S_\gamma))^{|\gamma|}
   \to \cE_\beta^\Delta \to M^s(n_\beta,d_\beta)\times \cN'\, .
$$
We have the natural inclusion $\cE_\beta^\Delta \hookrightarrow (\cE_\beta)^{|\beta|}$
as the smallest diagonal. Clearly, there is a
equivalence between partitions $\pi\leq \pi_0$ and diagonals: for any
$\pi\leq \pi_0$, let
\[
\cE^\Delta_\pi = \prod_{\beta\in \pi} \cE^\Delta_\beta \to \cN'\, ,
\]
where we have the obvious inclusion $\cE_\pi^\Delta\subset \tilde{X}^+(\bn)^\Delta$.
Hence
$$
\tilde{X}^+(\bn)=\tilde{X}^+(\bn)^\Delta \setminus \left(
\bigcup_{\pi\leq \pi_0} \cE^\Delta_\pi \right).
$$
Since we know that $[\tilde{X}^+(\bn)^\Delta/\frS_\bn] \in \RC$, the proof is complete as long as we show that
$$
 \left( \bigcup_{\pi\leq \pi_0} \cE^\Delta_\pi \right)/\frS_\bn
$$
lies in $\RC$. This is a stratified space, so we check that each
stratum is motivated by $C$. So, fix $\pi\leq \pi_0$ and consider
 $\frS_\bn \cdot \pi=\{ g\cdot \pi \, ; \, g\in \frS_\bn\}$ the orbit of $\pi$ under $\frS_\bn$. The corresponding
 stratum in the quotient is
  \begin{equation}\label{eqn:hoy4}
 \left( \bigcup_{\pi'\in \frS_\bn\cdot\pi} \cE_{\pi'}^\Delta \right) / \frS_\bn
 = \cE_\pi^\Delta / \text{Stab}_{\frS_\bn}(\pi) \, ,
  \end{equation}
where
 $$
 \text{Stab}_{\frS_\bn}(\pi) = \{ g\in \frS_\bn \, ; \, g\cdot \pi=\pi \}.
 $$
So it is enough to see that (\ref{eqn:hoy4}) is motivated by $C$.

Consider now a partition $P$ of the set $\pi$ (this is \emph{not} a partition of $[b]$),
defined as follows. If $\beta_1,\beta_2\in \pi$, then
 $$
 \beta_1 \sim \beta_2 \iff |\beta_1\cap \delta|=|\beta_2\cap \delta|, \ \text{for
 all } \delta \in \pi_1.
 $$
It is easy to see that we have an exact sequence of
groups
 $$
 1 \to \frS_{\pi\wedge \pi_1} \to \text{Stab}_{\frS_\bn}(\pi) \to \frS_P \to 1 \, ,
 $$
where $\frS_P= \prod_{p\in P} \frS_p$, and the projection $\text{Stab}_{\frS_\bn}(\pi) \to \frS_P$ associates to each $g\in \text{Stab}_{\frS_\bn}(\pi)$ the induced permutation of bricks in $P$.

The quotient (\ref{eqn:hoy4}) is then rewritten as
  \begin{equation}\label{eqn:hoy5}
  \cE_\pi^\Delta / \text{Stab}_{\frS_\bn}(\pi)  =
  (\cE_\pi^\Delta / \frS_{\pi\wedge \pi_1} )/\frS_P  \, .
  \end{equation}
From the definition of $\cE_\pi^\Delta$, we have
  \begin{equation}\label{eqn:hoy5b}
   \cE_\pi^\Delta / \frS_{\pi\wedge \pi_1}  =
   \prod_{\beta \in \pi} \cE_\beta^\Delta / \frS_{\pi\wedge \pi_1,\beta}\, ,
  \end{equation}
where $\frS_{\pi\wedge \pi_1,\beta}=\prod_{A\in\pi\wedge \pi_1}\frS_{A\cap\beta}$. Therefore the
quotient
 \begin{equation}\label{eqn:hoy6}
  \cE_\beta^\Delta / \frS_{\pi\wedge \pi_1,\beta} \to
  M^s(n_\beta,d_\beta)\x \cN'
  \end{equation}
is a fiber bundle with fiber
 $$
 \prod_{\substack{\gamma \in \pi\wedge\pi_1 \\ \gamma\subset\beta}} \Sym^{|\gamma|}
 \Gr(a_\gamma,\Ext^1(T',S_\gamma))\, .
 $$

Fix $p\in P$. So $p$ is a subset of $\pi$ and, from the definition of $P$, it follows that for
any $\beta\in p$, the space (\ref{eqn:hoy6}) is the same. Denote it by $\cE_p$ and consider it,
as before, as a bundle $\cE_p$ over $\cN'$.
Thus the quotient of \eqref{eqn:hoy5b} by $\frS_P$ is
  $$
  \left(\prod_{\beta \in \pi} \cE_p \right)/\frS_P = \left(\prod_{p\in P} (\cE_p)^{|p|}\right)/\frS_P
   = \prod_{p\in P} \Sym^{|p|} \cE_p\,
   $$
as a bundle over $\cN'$. This is the quotient \eqref{eqn:hoy5} and clearly its class in $K(\Mot)$ lies in $\RC$.

This finishes the proof of the proposition.
\end{proof}

\section{Two cases of strata $X^+(\bn)$ for $r=2$} \label{sec:r=2}

In this section we want to study some strata $X^+(\bn)\subset\cS_{\s_c^+}$ for which
the standard filtration has $r=2$, i.e., it is of the form $0\subset T_1\subset T_2\subset T_3=T$.
First we analyse the case where the type is
\begin{equation}\label{eqn:bn r=2 I}
\bn= \begin{pmatrix} (n_i) \\ \ba_{1} \\ \ba_{2} \end{pmatrix} =
\begin{pmatrix}
    n_1 & n_2 & \cdots & n_{b-1} & n_b \\
	0 & 0 & \cdots & 0 & 1 \\
	1 & 1 & \cdots & 1 & 0
\end{pmatrix}
\end{equation}
for some $n_i$, $b\geq 3$, and non-trivial $\frS_\bn$. Although we shall only need the case
$b=3$, $n_i=1$ for Theorem \ref{thm:ahora}, we will work out the
general case.

Recall that the basis of this stratum is
 \begin{equation}\label{eqn:un}
 \tilde{U}(\bn)= \left( \prod_{i=1}^b M^s(n_i,d_i) \setminus \Delta \right) \, ,
 \end{equation}
where $\Delta=\{(E_1,\ldots,E_b)\, ; \, E_i\cong E_j,\text{ for some }i,j\}$ is the big diagonal.

\begin{proposition} \label{prop:r=2.1}
Let $\sigma_c$ be any critical value.
Let $n\geq 1$, and $\bn$ be given by \eqref{eqn:bn r=2 I}.
Suppose that
$[M^s(n'',d'')]$ and $[\cN_{\sigma_c}^s(n',1,d',d_o)]$ are in $\RC$, for any $n',n''<n$,
$\gcd(n'',d'')=1$. Assume also that $\gcd(n_i,d_i)=1$, for every $i=1,\ldots,b$.
Then  $[X^\pm(\bn)]\in \RC$.
\end{proposition}

\begin{proof}
Again we only consider the case of $X^+(\bn)$. By Proposition \ref{prop:Xbn}, we have to
construct a two-step iterated fibration, with basis $\tilde{U}(\bn)\times\cN'$ where
$\cN'=\cN_{\s_c}^s(n',1,d',d_o)$, and then take the quotient
by the symmetric group $\frS_\bn$ which is a subgroup of the permutation
group on the first $b-1$ factors of \eqref{eqn:un}.

The first step is a bundle $X_1^+ \to X_0^+:=\tilde{U}(\bn)\times\cN'$ with fibers
 $$
 \prod_{i=1}^{b-1} \PP\Ext^1(T',S_i).
 $$
Recall that we write $S_i$ for the triple $(E_i,0,0)$ with $E_i\in M^s(n_i,d_i)$.

The second step is a bundle $X_2^+ \to X_1^+$. If $\widetilde{T}\in X_1^+$ is such
that $$0\to S_1\oplus\cdots\oplus S_{b-1}\to\widetilde{T}\to T'\to 0,$$ then, from
Proposition \ref{prop:Xbn}, the fiber of $X_2^+$ over $\widetilde T$ is
\begin{equation} \label{eqn:fibers}
 \PP \Ext^1(\widetilde{T},S_b) \setminus \bigcup_{i=1}^{b-1} \PP \Ext^1(\widetilde T_i, S_b)
 \end{equation}
where, for each $i$, $\widetilde T_i$ is the triple fitting in the natural exact sequence
 $$
 0\to S_i \to\widetilde{T}\to\widetilde{T}_i\to 0
 $$
so that
 $$
 0\to\Ext^1(\widetilde T_i,S_b)\to
 \Ext^1(\widetilde T,S_b)\to \Ext^1(S_i,S_b)\to 0.
 $$

\medskip
 We extend the fibration to a basis larger than \eqref{eqn:un}.
Consider for each $1\leq i\leq b-1$, the space
 $$
 Y_i:=(M^s(n_i,d_i) \times M^s(n_b,d_b) \setminus \Delta_i) \times \cN',
 $$
$\Delta_i=\{(E_i,E_b) \, ; \, E_i\cong E_b\}\subset
M^s(n_i,d_i) \times M^s(n_b,d_b)$, so that the fiber product
 $$
 \bar X_0^+ = \prod_{\substack{1\leq i \leq b-1 \\
  M^s(n_b,d_b) \times \cN'}} Y_i
 $$
consists of bundles $(E_1,\ldots, E_{b-1},E_b)$, where $E_i\not\cong E_b$, for $i\neq b$.
The fibration $X_1^+ \to X_0^+$ extends to a fibration $\bar X_1^+ \to \bar X_0^+$.
The dimension of $\Ext^1(\widetilde{T},S_b)$ stays constant as we move over each $Y_i$,
so the fibration $X_2^+ \to X_1^+$ extends to a fibration
$\bar X_2^+ \to \bar X_1^+$ with fibers $\PP \Ext^1(\widetilde{T},S_b)$,
also compactifying the fibers (\ref{eqn:fibers}).

The action of $\frS_\bn$ extends to $\bar X_1^+$. To work out the quotient, consider
again the partition $\pi_0$ of $[b-1]$ given by $i \sim j \iff n_i=n_j$.
Let $\cE_i$ be the bundle
 $$
 \PP \Ext^1(T',S_i) \to \cE_i \to Y_i
 $$
and consider it as a bundle over $B=M^s(n_b,d_b)\times \cN'$. Then
 $$
 \bar{X}_1^+= \prod_B \cE_i =\prod_{\substack{\alpha\in \pi_0\\ B}} (\cE_\alpha)^{|\alpha|}\, .
 $$
The quotient by $\frS_\bn$ is
 \begin{equation} \label{eqn:otro-hoy1}
 \bar{X}_1^+/\frS_\bn =\prod_{\substack{\alpha\in \pi_0\\ B}} \Sym^{|\alpha|} \cE_\alpha \, ,
   \end{equation}
and $\bar{X}_2^+/\frS_\bn \to \bar{X}_1^+/\frS_\bn$ is a projective bundle with fibers
$\PP \Ext^1 ( \widetilde{T},S_b)$ (this family is locally trivial in the Zariski topology
since it is the projectivization of a vector bundle).

We can construct a family of triples parametrized by (\ref{eqn:otro-hoy1})
using that $\gcd(n_i,d_i)=1$, for all $i$.
This family is clearly Zariski locally trivial.
The family of $S_b$ over $M^s(n_b,d_b)$ is also Zariski locally trivial
since $n_b$ and $d_b$ are coprime. So
 $$
 [\bar{X}_2^+/\frS_\bn]=[\bar{X}_1^+/\frS_\bn]\cdot [\PP^N] \in \RC\,.
 $$

\medskip
Now we will deal with the diagonals and the sub-fibrations.
We have fibrations 
 $$
  \PP \Ext^1 (\widetilde{T},S_b) \to \bar{X}_2^+ \to \bar{X}_1^+
  \qquad \text{and} \qquad \bar{X}_1^+ \to \prod_{\substack{1\leq i \leq b-1 \\
  M^s(n_b,d_b) \times \cN'}} Y_i.
 $$
The space $\bar X_2^+$ is thus a stratified space, where the strata are given according to the various
diagonals inside $\Delta$ in \eqref{eqn:un} and according to the sub-fibrations  with fibers $\PP \Ext^1(\widetilde{T}_I, S_b)$ (cf. \eqref{eqn:fibers}),
where $I\subset [b-1]$, $S_I:=\bigoplus_{i\in I} S_i$ and
  $$
  0\to  S_I\to \widetilde{T} \to\widetilde{T}_I\to 0.
  $$
Note that from this we have an extension
$0\to  S_{I^c} \to \widetilde{T}_I \to T'\to 0$, with $I^c:=[b-1]\setminus I$.

If all strata induced in the quotient $\bar{X}_1^+/\frS_\bn$ are in $\RC$, then the main
open set $X_1^+/\frS_\bn$, which is the stratum $X^+(\bn)$ we are dealing with, lies also in $\RC$.
Let us prove that every stratum in $\bar{X}_1^+/\frS_\bn$ is motivated by $C$, thus completing the proof.

The diagonals of (\ref{eqn:un}) are labeled by partitions $\pi\leq \pi_0$ and the sub-fibrations
are labeled by sets $I\subset [b-1]$. Let $\pi_I$ be the partition $\{I, I^c\}$. The stratum
$S_{(\pi,I)}$
associated to the pair $(\pi,I)$ is constructed as follows. For each $\beta\in\pi$, let
 $$
 Y_\beta=(M^s(n_\beta,d_\beta) \times M^s(n_b,d_b) \setminus \Delta_\beta) \times \cN'
 $$
where $n_\beta=n_i$ and $d_\beta=d_i$, for any $i\in\beta$. Then we have a fibration $X^+_\pi$
over the basis
 \begin{equation}\label{eqn:aaaa}
\prod_{\substack{\beta \in\pi \\ M^s(n_b,d_b)\times\cN'}} Y_\beta
 \end{equation}
whose fiber over $((E_\beta)_{\beta\in\pi} , E_b,T')\in Y_\beta$ is
 $$
\prod_{\beta\in\pi}\PP\Ext^1(T',S_\beta)^{|\beta|}
 $$
and then a second fibration over $X^+_\pi$ with fiber $\PP\Ext^1(\widetilde{T}_I,S_b)$ over
$\widetilde T_I\in X^+_\pi$.

The group $\frS_\bn$ acts on the pairs $(\pi,I)$, with orbit $\frS_\bn \cdot (\pi,I)$ and the
corresponding stratum in the quotient is
 $$
  \left( \bigcup_{(\pi',I')\in\frS_\bn \cdot (\pi,I)} S_{(\pi',I')} \right) /\frS_\bn =
 S_{(\pi,I)} /\text{Stab}_{\frS_\bn} (\pi,I)\,
 $$
where $\text{Stab}_{\frS_\bn}(\pi,I)$ is the stabilizer of the pair $(\pi,I)$.
 Hence we have to prove that
$$
 [S_{(\pi,I)} /\text{Stab}_{\frS_\bn} (\pi,I)]\in\RC.
$$

Now, there is an exact sequence
 $$
1\to \frS_{\pi\wedge \pi_I} \to \text{Stab}_{\frS_\bn} (\pi,I) \to \frS_P \to 1 ,
 $$
where $P$ is the partition of the set $\pi$ given by
 $$
 \gamma_1 \sim \gamma_2 \iff |\gamma_1\cap I|=|\gamma_2\cap I|, \ |\gamma_1\cap I^c|=|\gamma_2\cap I^c|, \
 $$
for $\gamma_1,\gamma_2\in \pi$, and $ \frS_P=\prod_{p\in P} \frS_p$. Then
 \begin{equation}\label{eqn:bbbb}
 X^+_\pi/\frS_{\pi\wedge\pi_I}
 \end{equation}
is a fibration over (\ref{eqn:aaaa}) with fiber
  $$
 \prod_{\delta \in \pi\wedge \pi_I} \Sym^{|\delta|} \left(\PP\Ext^1(T',S_\delta)\right).
 $$

Let $\beta \in \pi$.
\begin{itemize}
 \item If $\beta\subset I$, then we have a fibration
$\Sym^{|\beta|} (\PP\Ext^1(T',S_\beta)) \to \cE_\beta \to X_\beta$.
 \item If $\beta\subset I^c$, then we have a fibration
$\Sym^{|\beta|} (\PP\Ext^1(T',S_\beta)) \to \cE_\beta \to X_\beta$.
 \item If $\beta=\beta_1\cup \beta_2$, where $\beta_1=\beta\cap I$,
$\beta_2=\beta\cap I^c$ are both non-trivial, then the fibration is
$\Sym^{|\beta_1|} (\PP\Ext^1(T',S_\beta)) \times
 \Sym^{|\beta_2|} (\PP\Ext^1(T',S_\beta))
 \to \cE_\beta \to X_\beta$.
\end{itemize}
Then (\ref{eqn:bbbb}) can be rewritten as
 $$
 \prod_{\substack{\beta\in \pi\\ B}} \cE_\beta = \prod_{\substack{p\in P\\ B}} (\cE_p)^{|p|}
 $$
(recall $B=M^s(n_b,d_b)\times \cN'$) and
 $$
 (X^+_\pi/\frS_{\pi\wedge\pi_I})/\frS_P=
 \prod_{\substack{p\in P\\ B}} \Sym^{|p|} \cE_p \, .
 $$
 So
$$
 [X^+_\pi/\text{Stab}_{\frS_\bn} ]\in \RC.
 $$

Finally, the second fibration has as fiber a projective space, and
this is a product at the level of $K$-theory.
\end{proof}

Now we move to the case
\begin{equation}\label{eqn:bn r=2 II}
\bn= \begin{pmatrix} (n_i) \\ \ba_{1} \\ \ba_{2} \end{pmatrix} =
\begin{pmatrix}
    n_1 & n_2 & \cdots & n_{b-1} & n_b \\
	1 & 1 & \cdots & 1 & 0 \\
	0 & 0 & \cdots & 0 & 1
\end{pmatrix}
\end{equation}
for some $n_i$, $b\geq 3$, and non-trivial $\frS_\bn$. This case is slightly easier than the previous one. For Theorem \ref{thm:ahora}, we only need $b=3$ and $n_i=1$, but again we do the general case.
The basis of the stratum is also given by \eqref{eqn:un}, where $\Delta$ is the big diagonal.

\begin{proposition}\label{prop:r=2.2}
Let $\sigma_c$ be any critical value.
 Let $n\geq 1$, and $\bn$ be given by \eqref{eqn:bn r=2 II}.
Suppose that
$[M^s(n'',d'')]$ and $[\cN_{\sigma_c}^s(n',1,d',d_o)]$ are in $\RC$, for any $n',n''<n$,
$\gcd(n'',d'')=1$. Assume also that $\gcd(n_i,d_i)=1$, for every $i=1,\ldots,b$. Then  $[X^\pm(\bn)]\in \RC$.
\end{proposition}

\begin{proof}
We only work out $X^+(\bn)$.
 By Proposition \ref{prop:Xbn}, we have to
construct a two-step iterated fibration, with base $\tilde{U}(\bn)\times\cN'$
where $\cN'=\cN_{\s_c}^s(n',1,d',d_o)$, and then take the quotient
by the symmetric group $\frS_\bn$ which is a subgroup of the permutation
group on the first $b-1$ factors of (\ref{eqn:un}).

The first step is a bundle $X_1^+ \to \tilde{U}(\bn)\times\cN'$ with fibers
 $$
 \PP \Ext^1(T',S_b),
 $$
and the second step is a bundle $X_2^+ \to X_1^+$ whose fibers are the spaces
 $$
 \prod_{i=1}^{b-1} \PP \Ext^1(\widetilde{T},S_i)\setminus \prod_{i=1}^{b-1} \PP \Ext^1(T',S_i),
 $$
where $\widetilde{T}$ is the triple corresponding to the point in $X_1^+$
determined by the extension
 $$
 0 \to S_b \to \widetilde{T} \to T' \to 0\, .
 $$
We compactify the fibration to a fibration with fibers
$$
 \prod_{i=1}^{b-1} \PP \Ext^1(\widetilde{T},S_i).
$$
The quotient by $\frS_\bn$ is done in exactly the same fashion as that carried out
in Section \ref{sec:r=1}, changing the role of $\cN'$ to that of $X_1^+$.

The other stratum is worked out in the same fashion, as a fibration
over $X_1^+$ with fiber
$$
 \prod_{i=1}^{b-1} \PP \Ext^1(T',S_i).
$$
So finally $[X_2^+/\frS_\bn]=[X^+(\bn)]\in \RC$.
\end{proof}

\section{Some cases with non-coprime rank and degree}

So far, we have analysed cases where $\gcd(n_i,d_i)=1$, for all $i=1, \ldots, b$.
If $\gcd(n,d)=1$, then there exists a universal bundle $\cE \to M^s(n,d) \times C$,
and hence the projective bundle $\cU_m \to M^s(n,d)$ is locally trivial in the
Zariski topology.
In the case where $\gcd(n,d)>1$ such a universal bundle does not exist, and we have
to circumvent the situation in another way.

\begin{proposition}\label{prop:non-coprime2}
 Let $\sigma_c$ be any critical value and let $\bn$ be a type with $r=1$, $\gcd(n_1,d_1)>1$, $a_1=1$, and $\gcd(n_i,d_i)=1$, for $i=2, \ldots,b$.
 Assume that $\frS_\bn=\{1\}$.
 Define $\cD_m(\tilde n,1,\tilde d,d_o)$ as in \eqref{eqn:Dm}. Suppose that
$[M^s(n'',d'')]$, $[\cN_{\sigma_c}^s(n',1,d',d_o)]$ and $[\cD_m(\tilde n,1,\tilde d,d_o)]$ are in $\RC$, for any $n',n'',\tilde n<n$, $\gcd(n'',d'')=1$.
 Then $[X^\pm(\bn)]\in \RC$.
\end{proposition}

\begin{proof}
We deal with the case $X^+ (\bn)$, the other one being similar.
The space $X^+(\bn)$ is a fibration over
 $$
  \cM(\bn)= M^s(n_1,d_1) \x \left( \prod_{i=2}^b M^s(n_i,d_i) \setminus \Delta \right) \x \cN'\, ,
 $$
$\Delta$ denoting a suitable diagonal, with fiber
 $$
\PP\Ext^1(T',S_1) \x \prod_{i=2}^b \Gr(a_i, \Ext^1(T',S_i))\, .
 $$
As $\gcd(n_i,d_i)=1$, for $i\geq 2$, we know that there are universal bundles over $M^s(n_i,d_i)$.
The same happens for $\cN'= \cN_{\s_c}(n',1,d',d_o)$.
So the bundle over $\cM(\bn)$ with fiber $\prod_{i=2}^b \Gr(a_i, \Ext^1(T',S_i))$ is
Zariski locally trivial, and hence the motives are multiplicative.
So we shall assume from now on that $b=1$.

For $T'\in \cN'$, we consider the subset $X^+_{T'}\subset X^+(\bn)$ 
corresponding to fixing $T'$. Clearly $[X^+(\bn)]=[X^+_{T'}]\, [\cN']$,
so we have to see that $X^+_{T'}$ is motivated by $C$.

The map $X^+_{T'} \to M^s(n_1,d_1)$ is a fibration with fiber $\PP\Ext^1(T',S_1)$.
This projective fibration defines a Brauer class 
 \begin{equation*}
 \mathrm{cl}(X^+_{T'}) \in \mathrm{Br}(M^s(n_1,d_1)).
 \end{equation*}
Now consider the fibration $\cU_m =\cU_m(n_1,1,d_1, L_e)\to M^s(n_1,d_1)$,
for some line bundle $L_e$ with $\deg L_e = e \ll 0$. It has fiber $\PP \Hom(L_e,L_1)$.
The same argument as in the proof of Proposition 3.2 in \cite{Brauer} shows that the
Brauer class
 \begin{equation*}
 \mathrm{cl}(\cU_m) \in \mathrm{Br}(M^s(n_1,d_1))
 \end{equation*}
satisfies 
 \begin{equation}\label{eqn:arreglo}
 \mathrm{cl}(\cU_m) = \pm \, \mathrm{cl}(X^+_{T'}).
 \end{equation}

Now consider the pull-back diagram 
 $$
\xymatrix{ \cA \ar[r]\ar[d] & X^+_{T'}\ar[d]^{g} \\
 \cU_m\ar[r]_(.3){f} & M^s(n_1,d_1)\, ,
}
 $$
where all maps are projective fibrations.
Consider the fibration $X^+_{T'} \to M^s(n_1,d_1)$. It has Brauer 
class $c_1=\mathrm{cl}(X^+_{T'}) \in\mathrm{Br}(M^s(n_1,d_1))$.
Then the fibration $\cA\to\cU_m$ is the pull-back under $f:\cU_m \to M^s(n_1,d_1)$, so it has Brauer class 
$f^*c_1 \in \mathrm{Br}(\cU_m)$.  Now, by \cite[p. 193]{Gabber}, there is an exact sequence
    $$\ZZ\cdot\mathrm{cl}(\cU_m)\to\mathrm{Br}(M^s(n_1,d_1))\xrightarrow{f^*}\mathrm{Br}(\cU_m)\to 0.$$
From this, and using (\ref{eqn:arreglo}), it follows that $f^*c_1=0$. This implies that $\cA\to\cU_m$ is Zariski locally trivial.
Similarly, using the pull-back under $g$, the Brauer class of $\cA \to X^+_{T'}$ is also trivial, so the fibration 
is Zariski locally trivial as well. 

The above implies that the motives satisfy $[\cA]=[\cU_m] \, [\PP^a]=[X^+_{T'}] \, [\PP^b]$, for some $a,b\geq 0$.
Hence
 $$
 [X^+_{T'}]=[\cU_m] \, [\PP^a]\, [\PP^b]^{-1} \, .
 $$
By our assumptions on $\cD_m$ and since $n_1<n$, it follows that $[\cU_m]\in \RC$. Hence
$[X^+_{T'}]\in\RC$, as required.

\end{proof}

Consider now the case
\begin{equation}\label{eqn:bn r=2 non-coprime}
\bn= \begin{pmatrix} (n_i) \\ \ba_{1} \\ \ba_{2} \end{pmatrix} =
\begin{pmatrix}
    n_1 & n_2 & \cdots & n_b \\
	0 & 1 & \cdots & 1 \\
	1 & 0 & \cdots & 0
\end{pmatrix}.
\end{equation}

\begin{proposition}\label{prop:non-coprime3}
 Let $\sigma_c$ be any critical value and let $\bn$ be given by \eqref{eqn:bn r=2 non-coprime},
 such that $\gcd(n_1,d_1)>1$ and $\gcd(n_i,d_i)=1$, for $i=2, \ldots,b$.
 Assume that $\frS_\bn=\{1\}$.
 Suppose that $[M^s(n'',d'')]$ and $[\cN_{\sigma_c}^s(n',1,d',d_o)]$ are in $\RC$, for any $n',n''<n$, $\gcd(n'',d'')=1$.
 Then $[X^\pm(\bn)]\in \RC$.
\end{proposition}

\begin{proof}
Since $r=2$, we have two steps. The first step $X_1^+$ sits as the total space of a fibration whose basis is
 $$
  \cM(\bn)= M^s(n_1,d_1) \x \left( \prod_{i=2}^b M^s(n_i,d_i) \setminus \Delta \right) \x \cN'\, ,
  $$
where $\cN'= \cN_{\sigma_c}^s(n',1,d',d_o)$, and whose fiber is
  $$
  \PP\Ext^1(T',S_1).
  $$
Then $X_1^+$ is a space in the situation covered by Proposition \ref{prop:non-coprime2}, hence it lies in $\RC$.
The second step is a bundle $X_2^+\to X_1^+$ whose fiber is
\begin{equation}\label{eqn:fiberstep2}
 \prod_{i=2}^{b} \PP \Ext^1(\widetilde{T},S_i)\setminus \prod_{i=2}^{b} \PP \Ext^1(T',S_i),
\end{equation}
 where $\widetilde{T}$ is given by the extension
 $ 0 \to S_1 \to \widetilde{T} \to T' \to 0$.
There is a universal bundle parametrizing triples over $\cN'$, another one
parametrizing triples $\widetilde{T}$ (because the triples in
$X_1^+$ are all $\sigma_c^+$-stable), and another
universal bundle over $M^s(n_i,d_i)$, for $i=2,\ldots, b$.
So the bundle over $X_2^+\to X_1^+$ with fiber \eqref{eqn:fiberstep2} is
Zariski locally trivial, and hence the motives are multiplicative.
So $[X^+(\bn)]=[X_2^+]\in \RC$.

The case of $[X^-(\bn)]$ is similar.
\end{proof}

Finally, we have to look also to the case
\begin{equation}\label{eqn:bn r=2 non-coprime II}
\bn= \begin{pmatrix} (n_i) \\ \ba_{1} \\ \ba_{2} \end{pmatrix} =
\begin{pmatrix}
    n_1 & n_2 & \cdots & n_b \\
	1 & 0 & \cdots & 0 \\
	0 & 1 & \cdots & 1
\end{pmatrix}.
\end{equation}

\begin{proposition}\label{prop:non-coprime4}
 Let $\sigma_c$ be any critical value and let $\bn$ be given by \eqref{eqn:bn r=2 non-coprime II},
 such $\gcd(n_1,d_1)>1$ and $\gcd(n_i,d_i)=1$, for $i=2, \ldots,b$.
 Assume that $\frS_\bn=\{1\}$.
 Suppose that $[M^s(n'',d'')]$ and $[\cN_{\sigma_c}^s(n',1,d',d_o)]$ are in $\RC$, for any $n',n''<n$, $\gcd(n'',d'')=1$.
 Then $[X^\pm(\bn)]\in \RC$.
\end{proposition}
\begin{proof}
Since $r=2$, we have two steps. The first step $X_1^+$ sits as the total space of a fibration whose basis is
 $$
  \cM(\bn)= M^s(n_1,d_1) \x \left( \prod_{i=2}^b M^s(n_i,d_i) \setminus \Delta \right) \x \cN'\, ,
  $$
where $\Delta$ is the suitable diagonal, $\cN'= \cN_{\sigma_c}^s(n',1,d',d_o)$, and whose fiber is
  $$
\prod_{i=2}^{b} \PP \Ext^1(T',S_i).
  $$
Then $X_1^+$ is a space in the situation covered by Proposition \ref{prop:non-coprime2}, hence it lies in $\RC$.

The second step is a bundle $X_2^+\to X_1^+$. If $\widetilde T\in X_1^+$ is given by $$0\to S_2\oplus\cdots\oplus S_b\to\widetilde{T}\to T'\to 0,$$ then the fiber of $X_2^+\to X_1^+$ over $\widetilde{T}$ is
\begin{equation} \label{eqn:fibers non-coprime}
 \PP \Ext^1(\widetilde{T},S_b) \setminus \bigcup_{i=1}^{b-1} \PP \Ext^1(\widetilde T_i, S_b)
 \end{equation}
where, for each $i$, $\widetilde T_i$ is the triple fitting in the natural exact sequence
 $$
 0\to S_i \to\widetilde{T}\to\widetilde{T}_i\to 0
 $$
so that
 $$
 0\to\Ext^1(\widetilde T_i,S_b)\to
 \Ext^1(\widetilde T,S_b)\to \Ext^1(S_i,S_b)\to 0.
 $$
Now, concerning the fibration $X_2^+\to X_1^+$, we are in a situation similar
to that of Proposition \ref{prop:non-coprime2}, the only difference being that
instead of having a projective fibration, we have a projective fibration minus
some sub-fibrations (which are also projetive). So, we conclude again that
$X_2^+\to X_1^+$ is Zariski locally trivial, and since $X_1^+$ is motivated
by $C$, we deduce using \eqref{eqn:fibers non-coprime} that $[X^+(\bn)]=[X_2^+]\in \RC$.

The case of $[X^-(\bn)]$ is similar.
\end{proof}


\section{Proof of the main results}

Now we complete the proof of the main results of the paper, Theorems \ref{thm:main-HC} and
\ref{thm:motivated}.

\begin{proposition} \label{prop:ahora}
Let $\sigma_c$ be any critical value.
Suppose $n\leq 4$, and let $\bn$ be any type except $\bn=\bn_0=((n),(1))$, $\sigma_c=\sigma_m$.
Suppose that $[\cN_\sigma(n',1,d',d_o)]$ are in $\RC$, for any $n'<n$. 
 Then the stratum $[X^\pm(\bn)]\in K(\Mot)$ lies in $\RC$.
\end{proposition}

\begin{proof}
Assume first that $\gcd(n_i,d_i)=1$, for all $i=1,\ldots,b$.
If the group $\frS_\bn$ is trivial, then Proposition \ref{prop:no-finite-group}
gives the result.
Now suppose that the group $\frS_\bn$ is non-trivial. If $r=1$, then
the general result we prove in Proposition \ref{prop:r=1} implies that $[X^\pm(\bn)] \in\RC$.
As $n\leq 4$, the remaining cases are just
\begin{itemize}
\item $r=2$, $b=3$, $n'=1$ and $(n_i)=(1,1,1)$, $\ba_1=(0,0,1)$, $\ba_2=(1,1,0)$;
\item $r=2$, $b=3$, $n'=1$ and $(n_i)=(1,1,1)$, $\ba_1=(1,1,0)$, $\ba_2=(0,0,1)$.
\end{itemize}
These situations are included in Propositions \ref{prop:r=2.1} and \ref{prop:r=2.2}.

Assume now that there exists some pair $(n_i,d_i)$ with $n_i$ and $d_i$ not coprime. Since $n\leq 4$, this situation can only occur in one of the following cases (where we always have $\frS_\bn$  trivial):
\begin{itemize}
\item $r=1$, $b=2$, $n'=1$ and $(n_i)=(2,1)$, $\ba_1=(1,1)$;
\item $r=1$, $b=1$, $n'=2$ and $(n_1)=(2)$, $\ba_1=(1)$;
\item $r=1$, $b=1$, $n'=1$ and $(n_1)=(3)$, $\ba_1=(1)$;
\item $r=2$, $b=2$, $n'=1$ and $(n_i)=(2,1)$, $\ba_1=(0,1)$, $\ba_2=(1,0)$;
\item $r=2$, $b=2$, $n'=1$ and $(n_i)=(2,1)$, $\ba_1=(1,0)$, $\ba_2=(0,1)$.
\end{itemize}
The first three possibilities are covered by Proposition \ref{prop:non-coprime2}, the fourth by Proposition \ref{prop:non-coprime3} and the last by Proposition \ref{prop:non-coprime4}.
\end{proof}

Taking the isomorphism \eqref{eqn:isom} into account, the following result covers Theorem \ref{thm:motivated}.

\begin{theorem} \label{thm:ahora}
Let $n\leq 4$. For any $\sigma$, the moduli spaces $\cN_\s^s(n,1,d,d_o)$ lie in $\RC$. The same holds for $M^s(n,d)$ whenever $\gcd(n,d)=1$.
\end{theorem}

\begin{proof}
By definition, $[\Sym^k C]\in \RC$, for any $k\geq 1$.
Also, by \eqref{eqn:wRC} and \eqref{eqn:RC}, we have $[\Jac C]\in \RC$. So the result is true for $n=1$.

Now we proceed inductively. Let $n'<n$ and assume that the motives of
$M^s(n',d')$ and $\cN_\sigma(n',1,d',d_o)$ lie in $\RC$.
Then Proposition \ref{prop:ahora} says that $X^\pm(\bn)$ is also motivated by $C$
(except for the case $\sigma_c=\sigma_m$, $\bn=\bn_0$).
So if $\sigma_c>\sigma_m$,
 $$
  [\cS_{\s_c^\pm}]=\Big[ \, \bigsqcup_{\bn} \, X^\pm(\bn) \Big] \in\RC.
  $$
In particular, since $\cS_{\sigma_M^-}=\cN_{\sigma_M^-}$, we see that $\cN_{\sigma_M^-}$
is motivated by $C$. From \eqref{eqn:flip}, we conclude that $[\cN_\s^s(n,1,d,d_o)]\in\RC$,
for any $\sigma>\sigma_m$.

On the other hand, for $\sigma_c=\sigma_m$ we have
 $$
  [\cD_m]=\Big[ \bigsqcup_{\bn\neq \bn_0} X^+(\bn) \Big] \in\RC.
  $$
Therefore also
  $$
  \cU_m=\cN_{\s_m^+}^s(n,1,d,d_o)\setminus\cD_m 
   $$
is in $\RC$.

Finally, if $\gcd(n,d)=1$, using the projection in (\ref{eqn:alfa}) and considering, $d/n-d_o>2g-1$,  we deduce that $M^s(n,d)$ also lies in $\RC$.
\end{proof}

Note that Theorem \ref{thm:ahora} also applies to critical values $\sigma=\sigma_c$.
Also, it is worth noticing that the result that the moduli spaces $M^s(n,d)$, for any
$n,d$ coprime, are motivated by $C$, has already been proved by Del Ba\~no in Theorems 4.5 and 4.11 of \cite{DelBano}.

\begin{corollary} \label{cor:Hodge-cN}
For a generic curve $C$, for generic $\sigma$, and for $n\leq 4$,
$\cN_{\s}(n,1,d,d_o)$ satisfies the Hodge conjecture.
\end{corollary}

\begin{proof}
Recall that for generic $\sigma$, $\cN_\s^s(n,1,d,d_o)=\cN_\s(n,1,d,d_o)$ is smooth and projective.
Thus the result follows by applying the map $\Theta$ defined in \eqref{eqn:Theta}, and
noting that $\RC\subset \ker\Theta$ by Proposition \ref{prop:generic}.
\end{proof}

The proof of the Hodge conjecture for $M(n,d)$, where $n,d$ are coprime
and $C$ is generic, is given in \cite[Corollary 5.9]{DelBano}.


A comment about the restriction $n\leq 4$ is in order. The main obstacle to
give the result for arbitrary rank is the geometrical analysis of the flip loci.

It is clear that $\tilde X^{\pm}(\bn)$ is in $\RC$ by induction on the rank.
However, we have to quotient by the finite group $\frS_\bn$. For a \emph{smooth projective}
variety $X$, the motive $h(X/\frS_\bn)$ is --- see \eqref{quotient motive} --- a
sub-motive of $h(X)$, and with this it would
follow that if $[X]\in \RC$ then $[X/\frS_\bn]\in\RC$. However, the same statement
does not hold for quasi-projective varieties (our spaces $\tilde X^{\pm}(\bn)$ are smooth
but quasi-projective). That is why we have to carry a finer
analysis of the flip loci.

An alternative route would be to work with stacks in the spirit of \cite{hein},
to prove that the motive of $\cN_\sigma$ is in $\RC$. This misses the geometrical
description, but it might be applied for arbitrary rank. 
Indeed, very recently, after the submission of this paper, we have been
communicated that this has been undertaken in \cite{MR}.


\begin{thebibliography}{MMMM}

\bibitem{A} \textsc{Arapura, D.}: {Motivation for Hodge cycles}.
\textsl{Adv. Math.}  \textbf{207} (2006), no.2, 762--781.

\bibitem{AK} \textsc{Arapura, D.; Kang, S-J.}: {Coniveau and the Grothendieck group of varieties}.
\textsl{Michigan Math. J.}  \textbf{54}  (2006) 611--622.

\bibitem{DelBano}\textsc{del Ba\~{n}o, S.}: On the Chow motive of some moduli spaces. \textsl{J. reine angew. Math.} \textbf{532} (2001) 105--132.


\bibitem{B-NA} \textsc{del Ba\~{n}o, S.; Navarro Aznar, V.}:
On the motive of a quotient variety, Collect. Math., 49 (1998), 203-226.

\bibitem{B} \textsc{Bertram, A.}: Stable pairs and stable parabolic pairs, \textsl{J. Algebraic
Geom.} \textbf{3} (1994), no. 4, 703--724.


\bibitem{BL} \textsc{Birkenhake, C.; Lange, H.}: \textsl{Complex abelian varieties.} Second edition.
Grundlehren der Mathematischen Wissenschaften, {302}.
Springer-Verlag, Berlin, 2004.


\bibitem{Brauer} \textsc{Biswas, I.; Logares, M.; Mu\~noz, V.}:
Brauer group of moduli spaces of pairs.
\textsl{Comm. Algebra.} \textbf{40} (2012), 1605--1617.

\bibitem{BP} \textsc{Biswas, I.; Paranjape, K.H.}: The Hodge conjecture for general Prym varieties.
\textsl{J. Algebraic Geom.}  \textbf{11}  (2002), 33--39.

\bibitem{Bit} \textsc{Bittner, F.}: The universal Euler characteristic for varieties of characteristic zero, \textsl{Compositio Math.} \textbf{140} (2004), 1011--1032.

\bibitem{BY} \textsc{Boden, H.U.; Yokogawa, K.}:
Rationality of moduli spaces of parabolic bundles.
\textsl{J. London Math. Soc.} \textbf{59} (1999), 461--478.

\bibitem{BD}\textsc{Bradlow, S.B.; Daskalopoulos}:
{Moduli of stable pairs for holomorphic bundles over Riemann
surfaces}. \textsl{Internat. J. Math.} \textbf{2} (1991) 477--513.

\bibitem{BGP}\textsc{Bradlow, S.B.; Garc\'{\i}a--Prada, O.}:
{Stable triples, equivariant bundles and dimensional reduction}.
\textsl{Math. Ann.} \textbf{304} (1996) 225--252.

\bibitem{BGPG}\textsc{Bradlow, S.B.; Garc\'{\i}a-Prada, O.;
Gothen, P.B}: {Moduli spaces of holomorphic triples over compact
Riemann surfaces}. \textsl{Math. Ann.} \textbf{328} (2004)
299--351.

\bibitem{Gabber}\textsc{Gabber, O.}: {Some theorems on Azumaya algebras, 
(in: Groupe de Brauer)}. \textsl{Lecture Notes in Math.} \textbf{844}, Springer, Berlin-New York, 1981, 129--209.

\bibitem{GP}\textsc{Garc\'{\i}a-Prada, O.}: {Dimensional reduction of
stable bundles, vortices and stable pairs}. \textsl{Internat. J.
Math.} \textbf{5} (1994) 1--52.

\bibitem{hein} \textsc{Garc\'{\i}a-Prada, O., Heinloth, J., Schmitt, A.}:
{On the motives of moduli of chains and Higgs bundles}. Preprint arXiv:1104.5558v1.


\bibitem{Gooo} \textsc{G\"ottsche, L.}: On the motive of the Hilbert scheme of points on a surface,
\textsl{Math. Res. Lett.} \textbf{8} (2001), no. 5-6, 613--627.


\bibitem{manin} \textsc{Manin, Yu.I.}:
Correspondences, motifs and monoidal transformations,
\textsl{Math. USSR Sb.} \textbf{6} (1968), 439-470.

\bibitem{MR} \textsc{Mozgovoy, S; Reineke, M.}:
Moduli spaces of stable pairs and non-abelian zeta functions of curves via wall-crossing. 
Preprint arXiv:1310.4991.

\bibitem{Mu1} \textsc{Mu\~{n}oz, V.}:
Hodge polynomials of the moduli spaces of rank 3 pairs,
\textsl{Geometriae Dedicata} \textbf{136} (2008) 17--46.

\bibitem{Mu2} \textsc{Mu\~{n}oz, V.}:
Torelli theorem for the moduli spaces of pairs,
\textsl{Math. Proc. Cambridge Phil. Soc.} \textbf{146} (2009) 675--693.
	
\bibitem{Mu-HS} \textsc{Mu\~{n}oz, V.}:
Hodge structures of the moduli space of pairs,
\textsl{Internat. J. Math.} \textbf{21} (2010) 1505--1529.

\bibitem{MOV1} \textsc{Mu\~{n}oz, V.; Ortega, D.; V\'{a}zquez-Gallo, M-J.}:
Hodge polynomials of the moduli spaces of pairs.
\textsl{Internat. J. Math.} \textbf{18} (2007) 695--721.


\bibitem{Sch} \textsc{Schmitt, A.}:
{A universal construction for the moduli spaces of decorated
vector bundles}. \textsl{Transform. Groups} \textbf{9} (2004)
167--209.

\bibitem{Scholl}\textsc{Scholl, A.}: {Classical motives (in:  Motives)}. Proc. Sympos. Pure Math., 55, Part 1, Amer. Math. Soc., Providence, RI, 1994, 163--187.

\bibitem{Tankeev} \textsc{Tankeev, S. G.}: On algebraic cycles on surfaces and Abelian varieties.
\textsl{Izv. Akad. Nauk SSSR Ser. Mat.}, \textbf{45} 2 (1981), 398--434.


\end{thebibliography}
\end{document}